\documentclass[a4paper,11pt]{amsart}
\usepackage[colorlinks, linkcolor=blue,anchorcolor=Periwinkle,
citecolor=blue,urlcolor=Emerald]{hyperref}
\usepackage[all]{xy}
\SelectTips{cm}{}
\usepackage{hyperref}
\usepackage{graphicx}
\usepackage{psfrag}
\usepackage{tikz}
\usepackage{tikz-cd}
\usepackage{mathtools}
\usepackage{amsmath,amssymb,tikz-cd}
\usepackage[export]{adjustbox}
\usepackage{calligra}

\usetikzlibrary{decorations.pathreplacing}
\usetikzlibrary{matrix,arrows}

\usepackage[french,english]{babel}

\numberwithin{equation}{subsection}
\newtheorem{theorem}[subsection]{Theorem}
\newtheorem{definition}[subsection]{Definition}
\newtheorem{classification-theorem}[subsection]{Classification Theorem}
\newtheorem{decomposition-theorem}[subsection]{Decomposition Theorem}
\newtheorem{proposition-definition}[subsection]{Proposition-Definition}
\newtheorem{definition-proposition}[subsection]{Definition-Proposition}
\newtheorem{example-definition}[subsection]{Example-Definition}
\newtheorem{periodicity-conjecture}[subsection]{Periodicity Conjecture}
\newtheorem{lemma}[subsection]{Lemma}
\newtheorem{proposition}[subsection]{Proposition}
\newtheorem{corollary}[subsection]{Corollary}

\newtheorem{example}[subsection]{Example}
\newtheorem{remark}[subsection]{Remark}

\newtheorem{Definition-Proposition}[subsection]{D\'efinition-Proposition}

\newcommand{\reminder}[1]{}

\renewcommand{\mod}{\mathrm{mod}}

\newcommand{\inj}{\mathrm{inj}}
\newcommand{\Mod}{\mathrm{Mod}\,}
\newcommand{\CM}{\mathrm{CM}}

\newcommand{\proj}{\mathrm{proj}\,}

\newcommand{\per}{\mathrm{per} }

\newcommand{\pvd}{\mathrm{pvd} }
\newcommand{\add}{\mathrm{add} }

\renewcommand{\Im}{\mathrm{Im} }

\newcommand{\tr}{\mathrm{tr}}
\newcommand{\cat}{\mathrm{cat}}

\newcommand{\sSet}{\mathrm{sSet}}

\newcommand{\dgcat}{\mathrm{dgcat}}
\newcommand{\Hqe}{\mathrm{Hqe}}

\newcommand{\pretr}{\mathrm{pretr} }

\newcommand{\im}{\mathrm{im} }
\renewcommand{\ker}{\mathrm{ker} }

\newcommand{\Q}{\mathbb{Q}}

\newcommand{\iso}{\xrightarrow{_\sim}}

\newcommand{\Id}{\mathrm{id}}

\newcommand{\Def}{\mathrm{def}\kern 0.1em}
\newcommand{\gl}{\mathrm{gldim}}
\newcommand{\dom}{\mathrm{domdim}}
\newcommand{\cogen}{\mathrm{cogen}}

\newcommand{\D}{\mathcal {D}}
\newcommand{\A}{\mathcal {A}}
\newcommand{\B}{\mathcal {B}}
\newcommand{\C}{\mathcal {C}}
\newcommand{\E}{\mathcal {E}}
\newcommand{\F}{\mathcal {F}}
\newcommand{\G}{\mathcal {G}}
\newcommand{\J}{\mathcal {J}}
\newcommand{\I}{\mathcal {I}}
\newcommand{\M}{\mathcal {M}}
\newcommand{\N}{\mathcal {N}}
\newcommand{\V}{\mathcal {V}}

\newcommand{\T}{\mathcal T}
\renewcommand{\P}{\mathcal P}
\newcommand{\X}{\mathcal X}
\newcommand{\Y}{\mathcal Y}

\newcommand{\Hom}{\mathrm{Hom}}

\newcommand{\Ext}{\mathrm{Ext}}

\newcommand{\End}{\mathrm{End}}

\renewcommand{\L}{\mathcal L}

\renewcommand{\phi}{\varphi}

\newcommand{\K}{\mathcal{K}}

\newcommand{\inc}{\mathrm{inc}}
\renewcommand{\tilde}[1]{\widetilde{#1}}

\renewcommand{\Q}{\mathcal{Q}}

\begin{document}
\title[Iyama--Solberg correspondence for exact dg categories]{Iyama--Solberg correspondence for exact dg categories}
\author[Xiaofa Chen] {Xiaofa Chen}
\address{University of Science and Technology of China, Hefei, P.~R.~China}
\email{cxf2011@mail.ustc.edu.cn}

\date{\today}


\begin{abstract} 
We generalize the notions of $d$-cluster tilting pair and $d$-Auslander exact dg category to $d$-precluster tilting triple and $d$-minimal Auslander--Gorenstein exact dg category. 
We give a bijection between equivalence classes of $d$-precluster tilting triples and equivalence classes of $d$-minimal Auslander--Gorenstein exact dg categories. Our bijection generalizes 
%the
Iyama--Solberg correspondence for module categories.
\end{abstract}

\maketitle
%\dedicatory{}%
%\commby{}%

\section{Introduction}
Let $k$ be a commutative artinian ring. 
An Artin $k$-algebra is a $k$-algebra which is finitely generated as a $k$-module.
Let $A$ be an Artin $k$-algebra. 
We denote by $\mod A$ the category of finitely generated right $A$-modules. 
Let $d\geq 1$ be an integer. 
Recall that an $A$-module $M$ is {\em $d$-cluster-tilting}~\cite{Iyama07a} if 
\begin{align*}
\add M&=\{X\in\mod A\mid\Ext^{i}(M,X)=0 \text{ for $1\leq i\leq d-1$} \}\\
&=\{Y\in\mod A\mid \Ext^{i}(Y,M)=0 \text{ for $1\leq i\leq d-1$} \}.
\end{align*}
Iyama \cite{Iyama07} generalized the classical Auslander correspondence \cite{Auslander71} to the so-called Auslander--Iyama correspondence between
\begin{itemize} 
\item[-] Morita classes of $(A, M)$, where $A$ is an Artin $k$-algebra and $M$ a $d$-cluster tilting $A$-module, and
\item[-] Morita classes of $d$-Auslander algebras $\Gamma$, i.e.~an Artin $k$-algebra $\Gamma$ such that 
\begin{equation}\label{def:Auslander}
\gl \Gamma\leq d+1\leq \dom \Gamma.
\end{equation}
\end{itemize}
The correspondence takes $(A,M)$ to $\Gamma=\End_{A}(M)$. Here $\dom\Gamma$ denotes the dominant dimension of $\Gamma$ and $\gl\Gamma$ denotes the global dimension of $\Gamma$.
Recall that the {\em dominant dimension} of an abelian category $\A$ is the largest $n\in \mathbb N\cup \{\infty\}$ such that any projective object in $\A$ has an injective coresolution whose first $n$ terms are also projective. 
The {\em dominant dimension} of $A$ is defined as the dominant dimension of $\mod A$. 

In \cite{Chen23a}, we extended Auslander--Iyama correspondence to the setting of exact dg categories. 
Let us explain this in more detail.  For a pair $(\P,\I)$, where $\P$ is a connective {\em additive} dg category and $\I$ is an additive dg subcategory of $\P$, and for $n\geq 0$, we denote by $\A^{(n)}_{\P,\I}$ 
the full dg subcategory of $\pretr(\P)$ consisting of the objects in 
\[
\P\ast\Sigma\P\ast\cdots\ast\Sigma^{n-1}\P\ast\Sigma^n\P\cap \ker\Ext^{\geq 1}(-,\I).
\]
The pair $(\P,\I)$ is {\em $d$-cluster tilting} if $H^0(\I)$ is covariantly finite in $H^0(\A^{(d)}_{\P,\I})$.
Following Gorsky--Nakaoka--Palu~\cite{GorskyNakaokaPalu23}, for $d\geq 0$, an exact dg category $\A$ is {\em $d$-Auslander} if the extriangulated category $H^0(\A)$ has enough projective objects and satisfies similar homological properties as in (\ref{def:Auslander}).
Then the equivalence classes of $d$-cluster tilting pairs $(\P,\I)$ correspond to the equivalence classes of connective $d$-Auslander exact dg categories.
We mention that for a $d$-cluster tilting pair $(\P,\I)$, which consists of $\I=\inj A$ for some Artin $k$-algebra $\Lambda$ and $\P=\add M$  for some $A$-module $M$, the module $M$ is $d$-cluster tilting if and only if the corresponding $d$-Auslander exact dg category is concentrated in degree zero and abelian.

The notion of {\em $d$-precluster tilting} subcategory was introduced in \cite{IyamaSolberg18} as a generalization of the notion of $d$-cluster tilting subcategory, cf.~Definition~\ref{def:d-precluster tilting}.
An $A$-module $M$ is {\em $d$-precluster-tilting} if $\add M$ is a $d$-precluster tilting subcategory of $\mod A$.
The notion of $d$-minimal Auslander--Gorenstein algebra is introduced in {\em loc.~cit} as a generalization of the $d$-Auslander property. 
An Artin $k$-algebra $\Gamma$ is {\em $d$-minimal Auslander--Gorenstein} if it satisfies 
\begin{equation*}\label{def:AuslanderGorenstein}
\Id_{\Gamma} \Gamma\leq d+1\leq \dom \Gamma
\end{equation*}
where $\Id_{\Gamma}\Gamma$ denotes the injective dimension of $\Gamma\in\mod\Gamma$.
Iyama--Solberg~\cite[Theorem 4.5]{IyamaSolberg18} showed the following
\begin{theorem}[Iyama--Solberg correspondence]\label{intro:classical} 
The map $(A, M)\mapsto \Gamma=\End_{A}(M)$ induces a bijective correspondence between 
\begin{itemize}
\item[-] Morita classes of $(A, M)$, where $A$ is an Artin $k$-algebra and $M$ is a $d$-precluster tilting $A$-module, and
\item[-] Morita classes of $d$-minimal Auslander--Gorenstein $k$-algebras $\Gamma$.
\end{itemize}
\end{theorem}

In this article, we continue the work in \cite{Chen23a} and generalize the above Iyama--Solberg correspondence to the setting of exact dg categories. Let $d\geq 0$ be an integer. Our key notion is the following
\begin{definition}[Definition~\ref{def:d-precluster tilting triple}]\label{intro:d-precluster tilting triple}
Let $\K$ be an additive dg category and 
$\I\subseteq \P\subseteq \K$ full additive dg subcategories.
The triple $(\K,\P,\I)$ is {\em $d$-precluster tilting} if the following properties hold 
\begin{itemize}
\item[(1)] $\P$ is connective;
\item[(2)] $H^0(\I)$ is covariantly finite in $H^0(\B_{\K,\P,\I}^{(d)})$ (cf.~below);
\item[(3)] $\K$ is extension-closed in $\tr(\K)$;
\item[(4)] $\Hom_{\tr(\K)}(\P,\Sigma^{\geq i}\K )=0$  and $\Hom_{\tr(\K)}(\K,\Sigma^{\geq i}\P )=0$ for any $i\geq 1$;
\item[(5)] for each object  $K$ in $\K$, there is a triangle in $\tr(\K)$ 
\[
K'\rightarrow P\rightarrow K\rightarrow \Sigma K'
\]
where $P\in \P$ and $K'\in\K$.
\end{itemize}
\end{definition}

For a triple $(\K,\P,\I)$, where $\K$ is an additive dg category and $\I\subset \P$ are connective additive dg subcategories of $\K$, and for $n\geq 0$, we denote by $\B^{(n)}_{\K, \P,\I}$ 
 the full dg subcategory of $\pretr(\K)$ consisting of the objects in 
\[
\K\ast\Sigma\K\ast\ldots\ast\Sigma^{n-1}\K\ast \Sigma^{n}\mathcal K\cap \ker^{\geq 1}\Ext(-,\I).
\]

\begin{theorem}[Theorem~\ref{thm:Iyama--Solberg correspondence}]\label{intro:Iyama--Solberg correspondence}
There is a canonical bijection between the following:
\begin{itemize}
\item[(1)] equivalence classes of connective exact dg categories $\A$ which are $d$-minimal Auslander--Gorenstein;
\item[(2)] equivalence classes of $d$-precluster tilting triples $(\K,\P,\I)$.
\end{itemize}
The bijection from $(1)$ to $(2)$ sends $\A$ to the triple $(\K,\P,\I)$ formed by  the full dg subcategory $\P$ on the projectives of $\A$, its full dg subcategory $\I$ of projective-injectives and the full dg subcategory $\K$ of $\D^b_{dg}(\A)$ consisting of the objects $K$  which are $(d+1)$-st syzygies of objects in $\A$ (in particular we have $\Ext^{\geq 1}(K,\P)=0$). The inverse bijection sends $(\K,\P,\I)$ to the $\tau_{\leq 0}$-truncation of the  dg subcategory $\B^{(d+1)}_{\K,\P,\I}$
 of $\pretr(\K)$.
\end{theorem}

We refer to Definition~\ref{def:equ} for the definition of the equivalence relations mentioned in Theorem~\ref{intro:Iyama--Solberg correspondence}. Indeed, the framework of exact dg categories allows us to enhance the correspondence to an equivalence of $\infty$-groupoids, cf.~Theorem~\ref{thm:infinitygroupoids}.

We give a necessary and sufficient condition for the exact dg category $\tau_{\leq 0}\B_{\K,\P,\I}^{(d+1)}$ to be concentrated in degree zero (resp.~to be abelian), cf.~Proposition~\ref{prop:Quillenexact} (resp.~Proposition~\ref{prop:Abelian}). We show in Subsection~\ref{subsec:IS} that  the dg category $\A$ in Theorem~\ref{intro:Iyama--Solberg correspondence}~(1) is the module category over an Artin algebra $\Gamma$ (so in particular $\A$ is concentrated in degree zero) if and only if the corresponding triple $(\K,\P,\I)$  satisfies $\P=\add M\subset \mod\Lambda$ for a $d$-precluster tilting module $M$ over some Artin algebra $\Lambda$, and $\I$ is the full subcategory of injective $\Lambda$-modules and $\K$ is the full dg subcategory of $\D^b_{dg}(\mod\P)$ consisting of all $(d+1)$-st syzygies of $\P$-modules. In particular, our Theorem~\ref{intro:Iyama--Solberg correspondence} generalizes Theorem~\ref{intro:classical}. 
In Subsection~\ref{subsec:Grevstad}, we also compare our results with the correspondence obtained by Grevstad in~\cite{Grevstad22}.
\subsection*{Conventions}
In this article, we fix a commutative ring $k$. 
A dg $k$-category $\A$ is {\em additive} if $H^0(\A)$ is an additive category.
It is {\em connective} (resp.~{\em strictly connective}) if $H^i\A(A,B)=0$ (resp.~$\A^i(A,B)=0$) for $i>0$ and objects $A$ and $B$ in $\A$.
Additive subcategories of additive categories are assumed to be closed under direct summands.
For an additive $k$-category $\C$, we denote by $\Mod \C$ the category of $k$-linear functors $\C^{op}\rightarrow k\mbox{-}\Mod$. 
For each object $C\in\C$, we denote by $C^{\wedge}$ the representable right $\C$-module $\Hom_{\C}(?,C)\in \Mod\C$.
We denote by $\mod \C$ the category of finitely presented right $\C$-modules.
Dually, we denote by $\C\mbox{-}\Mod$ (resp.~$\C\mbox{-}\mod$) the category of $k$-linear functors $\C\rightarrow k\mbox{-}\Mod$ (resp.~the category of finitely presented left $\C$-modules).  
 For a subcategory $\D$ of $\C$, a morphism $f:X\rightarrow Y$ in $\C$ is a {\em left $\D$-approximation} of $X\in \C$ if $Y\in \D$ and the induced map $\Hom_{\C}(Y,Z)\rightarrow \Hom_{\C}(X,Z)$ is surjective for each $Z\in \D$. 
The full subcategory $\D$ is {\em cavariantly finite} if any $X\in \C$ has a left $\D$-approximation.
If $\X$ is a class of objects in $\C$, we write ${^{\perp}\X}\coloneqq\{C\in\C\mid \Hom_{\C}(C,X)=0\text{ for any $X\in\X$}\}$.
For a dg $k$-category $\P$, we denote by $\pretr(\P)$ its {\em pretriangulated hull} \cite{BondalKapranov90, Drinfeld04, BondalLarsenLunts04} and by $\tr(\P)=H^0(\pretr(\P))$ its {\em triangulated hull}.
A full additive dg subcategory $\I\subseteq \P$ is {\em covariantly finite} if $H^0(\I)$ is a covariantly finite subcategory of $H^0(\P)$.
For an exact dg category $\A$, we denote by $\D^b_{dg}(\A)$ its {\em (bounded) dg derived category} and by $\D^b(\A)=H^0(\D^b_{dg}(\A))$ its {\em bounded derived category} \cite{Chen23}.
We denote by $\dgcat$ the category of small dg $k$-categories and by $\Hqe$ the localization of $\dgcat$ with respect to the quasi-equivalences \cite{Tabuada05}.

\subsection*{Acknowledgements}
The author thanks his Ph.~D.~supervisor Bernhard Keller and his Ph.~D.~co-supervisor Xiao-Wu Chen for their constant support and their encouragements. 
The author thanks Yann Palu, Mikhail Gorsky and Norihiro Hanihara for helpful discussions.
%\section{Preliminaries}

%\section{Preliminaries}
%\section{Connection to Quillen exact categories}
%For simplicity, we assume that all additive categories appearing are weakly idempotent complete. 
%In this section, we collect basic definitions and results on extriangulated categories and on exact dg categories. The main references are~\cite{NakaokaPalu19, Palu23, Chen23}.

\section{Iyama--Solberg correspondence}
Throughout this section, let $d\geq 0$ be an integer. 
\begin{definition}[{\cite[Definition 3.5]{GorskyNakaokaPalu23}}]
Let $(\C,\mathbb E,\mathfrak s)$ be an extriangulated category with enough projectives. We define its {\em dominant dimension} $\mathrm{dom.dim}(\C,\mathbb E,\mathfrak s)$  to be the largest integer $n$ such that for any projective object $P$, there exist $n$ $\mathfrak s$-triangles
\[
\begin{tikzcd}[row sep=small]
P=Y^0\ar[r,"f^0",tail]&I^0\ar[r,two heads]&Y^1\ar[r,dashed]&;\\
Y^1\ar[r,"f^1",tail]&I^1\ar[r,two heads]&Y^2\ar[r,dashed]&;\\
&\ldots&&\\
Y^{n-1}\ar[r,"f^{n-1}",tail]&I^{n-1}\ar[r,two heads]&Y^{n}\ar[r,dashed]&,
\end{tikzcd}
\]
with $I^k$ being projective-injective for $0\leq k\leq n-1$. 
If such an $n$ does not exist, we let $\mathrm{dom.dim}(\C, \mathbb E,\mathfrak s)=\infty$.
\end{definition}
For an object $X$ in an extriangulated category $(\C,\mathbb E,\mathfrak s)$ and $\I\subset \C$ a full additive subcategory consisting of injective objects, we say that $X$ admits {\em an $\I$-coresolution of length $l$} if there exist $\mathfrak s$-triangles
\begin{equation*}
\begin{tikzcd}
X\ar[r]& I^0\ar[r]& X^1\ar[r,dashed]&\;;
\end{tikzcd}
\end{equation*}
\begin{equation*}
\begin{tikzcd}
X^1\ar[r]& I^1\ar[r]& X^2\ar[r,dashed]&\;;
\end{tikzcd}
\end{equation*}
\begin{equation*}
\begin{tikzcd}
\ldots
\end{tikzcd}
\end{equation*}
\begin{equation*}
\begin{tikzcd}
X^{l-1}\ar[r]& I^{l-1}\ar[r]& X^{l}\ar[r,dashed]&\;,
\end{tikzcd}
\end{equation*}
where $I^{i}\in\I$ for $0\leq i\leq l-1$. 
We call $X^{i}$ the {\em $i$-th cosyzygy} of $X$ (with respect to $\I$).
%{\color{red} in Section 3, we may need to assume $d\geq 1$}

\begin{definition}
An extriangulated category $(\C,\mathbb E, \mathfrak s)$ is {\em $d$-minimal Auslander--Gorenstein} if
\begin{itemize}
\item[(a)] $\C$ has enough projectives;% and enough injectives ({\color{red} in the definition of $d$-Auslander category, we do not need to assume that it has enough injectives because it is automatic});
\item[(b)] each projective has injective codimension $\leq d+1$, i.e.~$\mathbb E^{d+2}(-,P)=0$ for each projective object $P$ in $\C$;%({\color{red}in $d$-Auslander category, we require the global dimension is $\leq d+1$});
\item[(c)] it has dominant dimension $\geq d+1$.
\end{itemize}
An exact dg category $\A$ is {\em $d$-minimal Auslander--Gorenstein} if $H^0(\A)$ is $d$-minimal Auslander--Gorenstein. 
\end{definition}
\begin{remark}
A Frobenius extriangulated category is $d$-minimal Auslander--Gorenstein for each $d\geq 0$.
\end{remark}

%Let $\A$ be a $d$-minimal Auslander--Gorenstein exact dg category.
%and then $\A=\P\ast\Sigma\P\ast\ldots\ast\Sigma^d\P\ast \Sigma^{d+1}\mathcal K\cap \ker^{\geq 1}\Ext(-,\I)$ for some $\mathcal K\subset\C_{dg}(\P)$ such that for any $K\in\mathcal K$, we have $\Ext^{\geq 1}(K,\P)=0$.  

%Let $X$ be an object in $\A$ with $\Ext^{\geq 1}(X,\P)=0$. Then the right dg $\P$-module $\Hom_{\A}(F(-),X)$ is CM dg $\P$-module.
For a pretriangulated dg category $\T$ and full dg subcategories $\X$ and $\Y$,
by abuse of notation, we write $\X\ast\Y$ as the full dg subcategory of $\T$ consisting of the objects $Z$ which admit triangles in the triangulated category $H^0(\T)$
\[
X\rightarrow Z\rightarrow Y\rightarrow \Sigma X
\] 
with $X\in\X$ and $Y\in\Y$.
 For a triple $(\K,\P,\I)$, where $\K$ is an additive dg category and $\I\subset \P$ are connective additive dg subcategories of $\K$, and for $n\geq 0$, we denote by $\B^{(n)}_{\K, \P,\I}$ 
 the full dg subcategory of $\pretr(\K)$ consisting of the objects in 
\[
\K\ast\Sigma\K\ast\ldots\ast\Sigma^{n-1}\K\ast \Sigma^{n}\mathcal K\cap \ker\Ext^{\geq 1}(-,\I).
\]
A priori, the dg subcategory $\B_{\K,\P,\I}^{(n)}$ is not stable under extensions in $\pretr(\K)$. 
Below we will show that  this is indeed the case when the triple $(\K,\P,\I)$ is $d$-precluster tilting and $n\leq d+1$.
\begin{definition}\label{def:d-precluster tilting triple}
Let $\K$ be an additive dg category and 
$\I\subseteq \P\subseteq \K$ full additive dg subcategories.
The triple $(\K,\P,\I)$ is {\em $d$-precluster tilting} if it has the following properties 
\begin{itemize}
\item[(1)] $\P$ is connective;
\item[(2)] $H^0(\I)$ is covariantly finite in $H^0(\B_{\K,\P,\I}^{(d)})$;
\item[(3)] $\K$ is extension-closed in $\tr(\K)$;
\item[(4)] $\Hom_{\tr(\K)}(\P,\Sigma^{\geq i}\K )=0$  and $\Hom_{\tr(\K)}(\K,\Sigma^{\geq i}\P )=0$ for any $i\geq 1$;
\item[(5)] for each object  $K$ in $\K$, there is a triangle in $\tr(\K)$ 
\[
K'\rightarrow P\rightarrow K\rightarrow \Sigma K'
\]
where $P\in \P$ and $K'\in\K$.
\end{itemize}
\end{definition}
Note that the dg category $\K$ is usually not connective and 
we could rewrite property $(5)$ in Definition~\ref{def:d-precluster tilting triple} as follows: $\K\subset \P\ast\Sigma\K$ in $\pretr(\K)$.
%Put
%\[
%\B_{\K,\P,\I}^{(d+1)}=\K\ast\Sigma\K\ast\ldots\ast\Sigma^d\K\ast \Sigma^{d+1}\mathcal K\cap \ker^{\geq 1}\Ext(-,\I).
%\]
We denote by $\L_{\K, \P,\I}$ (resp.~$\Q_{\K,\P,\I}$, $\J_{\K,\P,\I}$) the full dg subcategory of $\B_{\K,\P,\I}^{(d+1)}$ on the objects in the closure under kernels of retractions of the objects in $H^0(\K)$ (resp.~in $H^0(\P)$, $H^0(\I)$).
%Clearly the dg subcateogry $\L_{\K,\P,\I}$ is stable under kernels of retractions in 
%\[
%\pretr(\L_{\K,\P,\I})\iso \pretr(\K).
%\]
%Except the condition of having enough injectives,
 The dg category $\B_{\K,\P,\I}^{(d+1)}$  is a $d$-minimal Auslander--Gorenstein exact dg category as the following proposition shows.

\begin{proposition}\label{prop:preclustertilting}
Let $(\K,\P,\I)$ be a $d$-precluster tilting triple of dg categories.
Then $\B_{\K,\P,\I}^{(d+1)}$ is extension-closed in $\pretr(\K)$. 
We endow it with the inherited exact dg structure.
The following  holds for $\B_{\K,\P,\I}^{(d+1)}$.
\begin{itemize}
\item[(a)] $\Q_{\K,\P,\I}$ is the full dg subcategory of projectives and $\B_{\K,\P,\I}^{(d+1)}$ has enough projectives;
\item[(b)] each projective has injective codimension $\leq d+1$;
\item[(c)] $\J_{\K,\P,\I}$ is the full dg subcategory of projective-injectives and $\B_{\K,\P,\I}^{(d+1)}$ has dominant dimension $\geq d+1$.
\end{itemize}
\end{proposition}
\begin{proof}
Indeed, it is easy to verify that  $\P\ast\Sigma\P\ast\ldots\ast\Sigma^d\P\ast \Sigma^{d+1}\mathcal K$ is extension-closed in $\pretr(\K)$ by (3) and (4) in Definition~\ref{def:d-precluster tilting triple}: we have
\[
\Sigma^{d+1}\K\ast\Sigma^{i}\P\subseteq \Sigma^i\P\ast\Sigma^{d+1}\K,\;\; i\leq d+1,
\]
 and that it is equal to $\B_{\K,\P,\I}^{(d+1)}$: 
\begin{align*}&\K\ast\Sigma\K\ast\ldots\ast\Sigma^d\K\ast \Sigma^{d+1}\mathcal K\\
\subset &(\P\ast \Sigma\K)\ast \Sigma\K\ast\Sigma^2\K\ast\ldots\ast\Sigma^d\K\ast \Sigma^{d+1}\mathcal K\\
\subset &\P\ast \Sigma\K\ast\Sigma^2\K\ast\ldots\ast\Sigma^d\K\ast \Sigma^{d+1}\mathcal K\\
\subset &\ldots\\
\subset &\P\ast\Sigma\P\ast\ldots\ast\Sigma^d\P\ast \Sigma^{d+1}\mathcal K.
\end{align*} 
Then it is clear that  $\Q_{\K,\P,\I}$ is the full dg subcategory of projectives and $\B_{\K,\P,\I}^{(d+1)}$ has enough projectives.
It is also clear that $\J_{\K,\P,\I}$ is the full dg subcategory of projective-injectives.

%It is obvious that the objects in $\P$ are projective in $H^0(\B_{\K,\P,\I}^{(d+1)})$. 
%Then item (a) is clear from the fact that $\B_{\K,\P,\I}^{(d+1)}$ is equal to $\P\ast\Sigma\P\ast\ldots\ast\Sigma^d\P\ast \Sigma^{d+1}\mathcal K$.

Item (b) follows from property (4) and the definition of $H^0(\B_{\K,\P,\I}^{(d+1)})$.

Item (c) follows from property (2). Indeed, let $Q$ be an object in $\Q_{\K,\P,\I}$. Then we have objects $P$ and $S$ in $\P$ such that $P\iso Q\oplus S$. We have the following diagram in $\tr(\K)$
\[
\begin{tikzcd}
Q\ar[r,equal]\ar[d,tail]&Q\ar[d,tail]&\\
P\ar[r,tail,"f"]\ar[d,two heads]&I\ar[r,two heads]\ar[d,two heads]&U\ar[d,equal]\\
S\ar[r]&V\ar[r,two heads]&U
\end{tikzcd}
\]
where $f$ is a left $H^0(\I)$-approximation. 
Now the claim follows from the Horseshoe Lemma applied to the conflation in the bottom row.
\end{proof}

For a connective exact dg category $\A$, we denote by $\overline{\A}$ 
the full dg subcategory of $\D^{b}_{dg}(\A)$ consisting of the objects 
in the closure under kernels of retractions of $H^0(\A)$ in $\D^{b}(\A)$. 
\begin{definition}[{\cite[Definition 3.6]{Chen23a}}]\label{def:equ}
\begin{enumerate}
\item[1)] Two connective exact dg categories $\A$ and $\B$ are {\em equivalent} if there is a quasi-equivalence 
%$\D^b_{dg}(\A)\iso \D^b_{dg}(\B)$ inducing a quasi-equivalence 
$\overline{\A}\iso \overline{\B}$.
\item[2)] Two $d$-precluster tilting triples $(\K, \P,\I)$ and $(\K', \P',\I')$ are {\em equivalent} if there is a quasi-equivalence 
%\[
%\pretr(\P)\iso \pretr(\P')
%\]
% inducing quasi-equivalences 
 $\L_{\K,\P,\I}\iso \L_{\K',\P',\I'}$ which restricts to quasi-equivalences $ \Q_{\K,\P,\I}\iso \Q_{\K',\P',\I'}$ and $\J_{\K,\P,\I}\iso\J_{\K',\P',\I'}$.
 \end{enumerate}
\end{definition}
%So the following condition on the triple $(\K,\P,\I)$ is natural
%\begin{itemize}
%\item[(6)] $H^0(\B_{\K,\P,\I}^{(d+1)})$ has enough injectives ({\color{red}this happens if and only if it is $d$-minimal Auslander--Gorenstein)}.
%\end{itemize}

\begin{theorem} \label{thm:Iyama--Solberg correspondence}
There is a canonical bijection between the following
\begin{itemize}
\item[(1)] equivalence classes of connective exact dg categories $\A$ which are $d$-minimal Auslander--Gorenstein;
\item[(2)] equivalence classes of $d$-precluster tilting triples $(\K,\P,\I)$.
\end{itemize}
The bijection from $(1)$ to $(2)$ sends $\A$ to the triple $(\K,\P,\I)$ formed by  the full dg subcategory $\P$ on the projectives of $\A$, its full dg subcategory $\I$ of projective-injectives and the full dg subcategory $\K$ of $\D^b_{dg}(\A)$ consisting of the objects $K$  which are $(d+1)$-st syzygies of objects in $\A$ (in particular we have $\Ext^{\geq 1}(K,\P)=0$). The inverse bijection sends $(\K,\P,\I)$ to the $\tau_{\leq 0}$-truncation of the  dg subcategory $\B^{(d+1)}_{\K,\P,\I}$
 of $\pretr(\K)$.
\end{theorem}
We will prove Theorem~\ref{thm:Iyama--Solberg correspondence} after Proposition~\ref{prop:K}.
\begin{lemma}\label{lem:d-minimal AG exact dg category}
Let $\A$ be a connective $d$-minimal Auslander--Gorenstein exact dg category. 
Let  $\P=\P_{\A}$ (resp.~$\I=\I_{\A}$) be the full dg subcategory on the projectives (resp.~projective-injectives) of $\A$.
Let $\K=\K_{\A}$ be the full dg subcategory of $\D^b_{dg}(\A)$ consisting of the objects $K$, which are $(d+1)$-st syzygies of objects in $\A$.
Then the following statements hold:
\begin{itemize}
\item[(a)] The universal embedding $\A\rightarrow\D^b_{dg}(\A)$ induces a quasi-fully faithful morphism $\P\hookrightarrow \K$.
We view $\P$ as a full dg subcategory of $\K$ via this morphism.
\item[(b)] Each object $K$ in $\K$ admits triangles in $\D^b(\A)$
\begin{equation*}
\begin{tikzcd}
K\ar[r]& I^0\ar[r]& L^1\ar[r]&\Sigma K;
\end{tikzcd}
\end{equation*}
\begin{equation*}
\begin{tikzcd}
L^1\ar[r]& I^1\ar[r]& L^2\ar[r]&\Sigma L^1;
\end{tikzcd}
\end{equation*}
\begin{equation*}
\begin{tikzcd}
\ldots
\end{tikzcd}
\end{equation*}
\begin{equation*}
\begin{tikzcd}
L^d\ar[r]& I^d\ar[r]& L^{d+1}\ar[r]&\Sigma L^d,
\end{tikzcd}
\end{equation*}
where $I^{i}\in\I$ and $L^{i+1}\in \A$  for $0\leq i\leq d$.
In particular, by the Horseshoe Lemma, the category $H^0(\K)$ is extension-closed in $H^0(\D^b_{dg}(\A))$.
\item[(c)] We have $\K\subseteq \P\ast\Sigma\K$ in $\pretr(\K)$ and $\Ext^{\geq 1}(\K,\P)=0$ and $\Ext^{\geq 1}(\P,\K)=0$.
\item[(d)] $H^0(\I)$ is covariantly finite in $H^0(\B_{\K,\P,\I}^{(d)})$.
\item[(e)] $\A$ is equivalent to $\tau_{\leq 0}\B_{\K,\P,\I}^{(d+1)}$ in the sense of Definition~\ref{def:equ} 1), and $\pretr(\K)$ is quasi-equivalent to $\D^b_{dg}(\A)$ .
\end{itemize}
In particular, the triple $(\K,\P,\I)$ is $d$-precluster tilting.
%Let $(\K,\P,\I)$ be a triple in Theorem~\ref{thm:Iyama--Solberg correspondence}~(2).
\end{lemma}
\begin{proof}
%Since $H^0(\A)$ is $d$-minimal Auslander--Gorenstein, each object $P\in\P$ admits an $\I$-coresolution of length $d+1$. 
Statement (a) is straightforward.
 
Let $K$ be an object in $\K$. 
By definition, there exists triangles in $\D^b(\A)$
\begin{equation}\label{tri:K}
\begin{tikzcd}
K=M^0\ar[r]& P^0\ar[r]& M^1\ar[r]&\Sigma K;
\end{tikzcd}
\end{equation}
\begin{equation*}
\begin{tikzcd}
M^1\ar[r]& P^1\ar[r]& M^2\ar[r]&\Sigma L^1;
\end{tikzcd}
\end{equation*}
\begin{equation*}
\begin{tikzcd}
\ldots
\end{tikzcd}
\end{equation*}
\begin{equation}\label{tri:Md-1}
\begin{tikzcd}
M^{d-1}\ar[r]& P^{d-1}\ar[r]& M^{d}\ar[r]&\Sigma M^{d-1},
\end{tikzcd}
\end{equation}
\begin{equation}\label{tri:Md}
\begin{tikzcd}
M^d\ar[r]& P^d\ar[r]& M^{d+1}\ar[r]&\Sigma M^d,
\end{tikzcd}
\end{equation}
where $P^{i}\in\P$ and $M^{i+1}\in \A$  for $0\leq i\leq d$.
We prove by induction that the object $M^{i}$ has an $\I$-coresolution of length $d+1-i$ for $0\leq i\leq d+1$.
Then statement (b) will follow from the case when $i=0$.
For $i=d+1$, the statement is clearly true. 
We assume that the statement holds for $M^{i+1}$.
Since $H^0(\A)$ is $d$-minimal Auslander--Gorenstein, the object $P^i$ has an $\I$-coresolution of length $d+1$. 
So there is a triangle
\[
P^i\rightarrow I^{i}\rightarrow U^i\rightarrow \Sigma P^i
\]
where $U^i$ has $\I$-coresolution of length $d$ and we have the following diagram in $\D^b(\A)$
\[
\begin{tikzcd}
M^i\ar[r]\ar[d,equal]&P^i\ar[r]\ar[d]&M^{i+1}\ar[r]\ar[d]&\Sigma M^i\ar[d,equal]\\
M^i\ar[r]&I^i\ar[r]\ar[d]&V^i\ar[r]\ar[d]&\Sigma M^i\\
&U^i\ar[r,equal]&U^i&
\end{tikzcd}.
\]
By the Horseshoe Lemma, the object $V^i$ has $\I$-coresolution of length $d-i$.
It follows that the object $M^i$ has $\I$-coresolution of length $d+1-i$ and this finishes the proof of the induction step.

Let us show statement (c). 
We have a quasi-fully faithful morphism $\pretr(\K)\hookrightarrow \D^b_{dg}(\A)$. 
Therefore it is enough to show that $\K\subseteq \P\ast\Sigma \K$ in $\D^b_{dg}(\A)$.
Let $K$ be an object in $\K$ with triangles (\ref{tri:K}--\ref{tri:Md}) 
where $P^{i}\in\P$ and $M^{i+1}\in \A$  for $0\leq i\leq d$.
Since $H^0(\A)$ has enough projectives, we have a triangle in $\D^b(\A)$
\begin{equation}\label{tri:K'}
\begin{tikzcd}
K'\ar[r] &P^{-1}\ar[r]&K\ar[r]&\Sigma K'.
\end{tikzcd}
\end{equation}
Then the triangle (\ref{tri:K'}) and triangles (\ref{tri:K}--\ref{tri:Md-1}) show that $K'\in \K$ and hence $\K\subseteq \P\ast\Sigma \K$.
From the triangles (\ref{tri:K}--\ref{tri:Md}), we have $\Ext^{i}(M^{j},P)\iso\Ext^{i+1}(M^{j+1},P)$ for $i\geq 1$ and $P\in\P$ and $j\geq 0$.
It follows that $\Ext^i(K,P)\iso \Ext^{i+d+1}(M^{d+1},P)$ for $i\geq 1$ and $P\in\P$.
Since the injective codimensions of projectives are $\leq d+1$, it follows that $\Ext^{\geq 1}(K,P)=0$ for $P\in\P$.
Since $\P$ consists of projective objects in $\A$, we also have $\Ext^{\geq 1}(P,K)=0$ for $P\in\P$.

Let us show statement (e). 
We have the following diagram
\[
\begin{tikzcd}
&&&\A\ar[d,hook]\ar[llldd,bend right=4ex,dashed,"\sim"swap]\\
&&\K\ar[r,hook]\ar[d,hook]&\D^b_{dg}(\A)\\
\tau_{\leq 0}\B_{\K,\P,\I}^{(d+1)}\ar[r]&\B_{\K,\P,\I}^{(d+1)}\ar[r,hook]&\pretr(\K)\ar[ru,dashed,hook,"\sim"swap]&
\end{tikzcd}
\]
Indeed, we need to show 
\begin{itemize} 
\item[(1)] the canonical morphism $\pretr(\K)\hookrightarrow \D^b_{dg}(\A)$ is quasi-dense, and 
\item[(2)] $\A$ is contained in $\B=\B_{\K,\P,\I}^{(d+1)}$, and 
\item[(3)] under the quasi-fully faithful morphism $\A\rightarrow \tau_{\leq 0}\B$, each object in $\B$ is the kernel of a retraction in $\A$.
\end{itemize}

For each object $X$ in $\A$, we have conflations in $H^0(\A)$
\begin{equation}\label{conf:X}
\begin{tikzcd}[row sep=small]
X_1\ar[r]& P_0\ar[r]& X\ar[r,dashed,"\delta_0"]&;\\
X_2\ar[r]& P_1\ar[r]& X_1\ar[r,dashed,"\delta_1"]&;\\
&\ldots&&\\
X_{d+1}=K_{d+1}\ar[r]& P_d\ar[r]& X_d\ar[r,dashed,"\delta_d"]&,
\end{tikzcd}
\end{equation}
where $P_i\in H^0(\P)$ for each $0\leq i\leq d$.
It follows that $K_{d+1}\in\K$.
This shows item (2). Since $\P$ is contained in $\K$,  item (1) also follows.

Let us show item (3). For an object $X$ in $\B$, it admits conflations (\ref{conf:X}) where $P_i$ is projective for each $0\leq i\leq d$ and $K_{d+1}\in\K$.
 Let us show that $X$ is in the closure of kernels of retractions in $\A$.
 We prove by induction on $i$ that $X_{i}$ admits an $\I$-coresolution of length $i$ with the $i$-th cosygyzy an object in $\A$ for each $0\leq i\leq d+1$.
 It obviously holds for $X_{d+1}=K_{d+1}\in H^0(\A)$: By item (b) the object $K_{d+1}$ admits an $\I$-coresolution of length $d+1$ with the $(d+1)$-cosyzygy an object in $\A$.
 
 %Since $\A$ is $d$-minimal Auslander--Gorenstein, we have the following diagram
% \[
 %\begin{tikzcd}
%K= K_{d+1}\ar[r]\ar[d,equal]&P^0\ar[r]\ar[d]&M^1\ar[d]\\
% K\ar[r]&I^0\ar[r]\ar[d]&N^1\ar[d]\\
% &U^0\ar[r,equal]&U^0
% \end{tikzcd}
% \]
% where $I^0\in\I$ and $U^0\in \A$ admits an $\I$-coresolution of length $d$. 

 Let us assume that it holds for $i+1\geq 1$. So we have a conflation
 \[
 X_{i+1}\rightarrow I_{i+1}\rightarrow U_{i+1}
 \]
 where $I_{i+1}\in\I$ and $U_{i+1}\in H^0(\A)$ admits an $\I$-coresolution of length $i$ with the $i$-th cosygyzy an object in $\A$.
We have the following diagram
\[
\begin{tikzcd}
X_{i+1}\ar[r]\ar[d]& P_i\ar[r]\ar[d]& X_i\ar[d,equal]\\
I_{i+1}\ar[r]\ar[d]&V_{i+1}\ar[r]\ar[d]&X_i\\
U_{i+1}\ar[r,equal]&U_{i+1}&
\end{tikzcd}
\]
From the second column, it follows by the Horseshoe Lemma that 
$V_{i+1}\iso X_i\oplus I_{i+1}\in H^0(\A)$ admits an $\I$-coresolution of length $i$ 
with the $i$-th cosygyzy an object in $\A$.
A similar argument to that of Proposition~\ref{prop:preclustertilting} (c) 
shows that the same is true for $X_i$.
This finishes the proof of the induction step.
We apply the above argument to the case $i=0$ and 
it follows that $X\oplus I_1\iso V_1\in \A$ and 
hence $X$ is in the closure under kernels of retractions in $H^0(\A)$.
This shows item (3) and we are done.

Let us show statement (d). 
Let $N$ be an object in $\B^{(d)}_{\K,\P,\I}$.
By definition, it admits triangles in $\tr(\K)$
\begin{equation}\label{tri:N}
\begin{tikzcd}
K^0\ar[r]& K^1\ar[r]& N^1\ar[r]&\Sigma K^0;
\end{tikzcd}
\end{equation}
\begin{equation*}
\begin{tikzcd}
N^1\ar[r]& K^2\ar[r]& N^2\ar[r]&\Sigma N^1;
\end{tikzcd}
\end{equation*}
\begin{equation*}
\begin{tikzcd}
\ldots
\end{tikzcd}
\end{equation*}
\begin{equation}\label{tri:Nd}
\begin{tikzcd}
N^{d-1}\ar[r]& K^{d}\ar[r]& N=N^{d}\ar[r]&\Sigma N^{d-1},
\end{tikzcd}
\end{equation}
where $K^i\in \K$ for $0\leq i\leq d$ and $N^j\in \B^{(d)}_{\K,\P,\I}$ for $1\leq j\leq d$.
By item (b), the object $K^0\in\K$ admits an $\I$-coresolution of length $d+1$. 
So we have the following diagram in $\tr(\K)$
\[
\begin{tikzcd}
K^0\ar[r]\ar[d]&K^1\ar[r]\ar[d]&N^1\ar[r]\ar[d,equal]&\Sigma K^0\ar[d]\\
I\ar[r]\ar[d]&U\ar[r]\ar[d]&N^1\ar[r]&\Sigma I\\
V\ar[r,equal]&V&&
\end{tikzcd}
\]
where $V$ admits an $\I$-coresolution of length $d$. 
By the Horseshoe Lemma, the object $U\iso N^1\oplus I$ also admits an $\I$-coresolution of length $d$. 
So we replace $N^1$ by $N^1\oplus I$ and we repeat the above argument. 
We see that $N=N^d$ admits an inflation into an object in $\I$. 
Hence $H^0(\I)$ is covariantly finite in $H^0(\B_{\K,\P,\I}^{(d)})$.
\end{proof}
\begin{remark}\label{rmk:wic}
For a $d$-precluster tilting triple $(\K,\P,\I)$, we have that $\B_{\K,\P,\I}^{(d+1)}$ is stable under kernels of retractions in $\pretr(\K)$. The proof is similar to that of \cite[Proposition 3.10]{Chen23a}.
Via the quasi-equivalence $\pretr(\K_{\A})\rightarrow \D^b_{dg}(\A)$, the canonical morphism $\A\rightarrow \tau_{\leq 0}\B_{\K_{\A},\P_{\A},\I_{\A}}$ identifies $\tau_{\leq 0}\B_{\K_{A},\P_{\A},\I_{\A}}$ with $\overline{\A}$, cf.~Lemma~\ref{lem:d-minimal AG exact dg category} (e).
\end{remark}
For a $d$-precluster tilting triple $(\K,\P,\I)$, %in Theorem~\ref{thm:Iyama--Solberg correspondence}~(2), 
put $\B=\B^{(d+1)}_{\K,\P,\I}$ and $\B'=\tau_{\leq 0}\B$.
\begin{lemma}\label{lem:pretri}
Let  $(\K,\P,\I)$ be a $d$-precluster tilting triple. % in Theorem~\ref{thm:Iyama--Solberg correspondence}~(2).
Then we have a canonical quasi-equvalence $\D^b_{dg}(\B')\iso\pretr(\K)$.

\end{lemma}
\begin{proof}
By Proposition~\ref{prop:preclustertilting}, the dg category $\B$ is stable under extensions in $\pretr(\K)$ 
and inherits a canonical exact structure. 
Let $(H^0(\B),\mathbb E,\mathfrak s)$ be the canonical extriangulated structure.

We have the following diagram in $\Hqe$
\[
\begin{tikzcd}
\B'=\tau_{\leq 0}\B\ar[d,hook]\ar[r,hook] &\B\ar[r,hook] &\pretr(\K)\\
\D_{dg}^{b}(\B')\ar[rru,dashed]&&
\end{tikzcd}
\]
Let $X$ be an object in $\B'$. 
Then we have a conflation in $H^0(\B')$
\begin{equation}\label{conflation:X}
X_1\rightarrow P_0\rightarrow X
\end{equation}
where $P_0\in H^0(\P)$.
Let $Y$ be any object in $\B'$.
Then we have 
\[
\Hom_{\D^b(\B')}(X,\Sigma Y)\iso \Hom_{\tr(\K)}(X,\Sigma Y)
\]
since both are isomorphic to $\mathbb E(X,Y)$.
The conflation (\ref{conflation:X}) gives rise to triangles in $\tr(\K)$ and in $\D^b(\B')$.
We apply the functor $\Hom_{\D^b(\B')}(-,Y)$ to the triangle in $\D^b(\B')$ corresponding to the conflation (\ref{conflation:X}) and we obtain a long exact sequence
\begin{align*}
\Hom(X,Y)&\rightarrow \Hom(P_0,Y)\rightarrow \Hom(X_1,Y)\rightarrow \Ext^1(X,Y) \\&\rightarrow\Ext^1(P_0,Y)=0\rightarrow\Ext^1(X_1,Y)\rightarrow \Ext^2(X,Y)\rightarrow \Ext^2(P_0,Y)=0\rightarrow \cdots
\end{align*}
We also have a corresponding long exact sequence in $\pretr(\K)$.
Note that $\Ext^{\geq 1}_{\D^b(\B')}(P_0,Y)$ vanishes since $P_0$ is projective in $\B'$ and $\Ext^{\geq 1}_{\tr(\K)}(P_0,Y)$ vanishes by the definition of $d$-precluster tilting triple.
It follows that 
\[
\Hom_{\D^b(\B')}(X,\Sigma^2Y)\iso\Hom_{\tr(\K)}(X,\Sigma^2 Y)
\]
 and similarly for higher extension groups.
It follows that $\pretr(\K)$ is quasi-equivalent to $\D^b_{dg}(\B')$.
\end{proof}

\begin{proposition}\label{prop:K}
Let  $(\K,\P,\I)$ be a $d$-precluster tilting triple. % in Theorem~\ref{thm:Iyama--Solberg correspondence}~(2).
We identify $\pretr(\K)$ with $\D^b_{dg}(\B')$ by Lemma~\ref{lem:pretri}.
Then the closure of $\K$ under kernels of retractions in $\D^b_{dg}(\B')$ is the full dg subcategory of $\D^b_{dg}(\B')$ consisting of the objects $K$, which are (d+1)-th syzygies of the objects in $\B'$.
\end{proposition}
\begin{proof}
By Proposition~\ref{prop:preclustertilting}, the closure $\Q_{\K,\P,\I}$ (resp.~$\J_{\K,\P,\I}$) of the dg category $\P$ (resp.~$\I$) under kernels of retractions is the full dg subcategory of $\B'$ on the projectives (resp.~projective-injectives).

Let $K$ be an object in $\D^b(\B')$ which admits triangles (\ref{tri:K}--\ref{tri:Md}) where $P^i\in\Q_{\K,\P,\I}$ and $M^{i+1}\in \B'$  for $0\leq i\leq d$. We may assume that $\P^i$ lies in $\P$ for $0\leq i\leq d$.
By definition of $\B'$, the object $M^{d+1}$ admits triangles in $\D^b(\B')$
\begin{equation*}%\label{tri:K'}
\begin{tikzcd}
K'=N^0\ar[r]& Q^0\ar[r]& N^1\ar[r]&\Sigma K';
\end{tikzcd}
\end{equation*}
\begin{equation*}
\begin{tikzcd}
N^1\ar[r]& Q^1\ar[r]& N^2\ar[r]&\Sigma L^1;
\end{tikzcd}
\end{equation*}
\begin{equation*}
\begin{tikzcd}
\ldots
\end{tikzcd}
\end{equation*}
\begin{equation*}\label{tri:Md-1'}
\begin{tikzcd}
N^{d-1}\ar[r]& Q^{d-1}\ar[r]& N^{d}\ar[r]&\Sigma N^{d-1},
\end{tikzcd}
\end{equation*}
\begin{equation*}\label{tri:Md'}
\begin{tikzcd}
N^d\ar[r]& Q^d\ar[r]& M^{d+1}\ar[r]&\Sigma N^d,
\end{tikzcd}
\end{equation*}
where $Q^{i}\in\P$ and $N^{i}\in \B'$  for $0\leq i\leq d$, and $K'\in\K$.
Since objects in $\P$ are projective in $H^0(\B')$, we have the following diagram
\[
\begin{tikzcd}
M^d\ar[r]\ar[d,dashed]&P^d\ar[r]\ar[d,dashed]&M^{d+1}\ar[d,equal]\ar[r]&\Sigma M^d\ar[d]\\
N^d\ar[r]&Q^d\ar[r]&M^{d+1}\ar[r]&\Sigma N^d
\end{tikzcd}
\]
where the sequence
\[
\begin{tikzcd}
M^d\ar[r]&N^d\oplus P^d\ar[r]&Q^d
\end{tikzcd}
\]
is a conflation and is thus split. We repeat the argument and we see that we have an isomorphism $K'\oplus P\iso K\oplus Q$ for some  objects $P$ and $Q$ in $\P$.
Hence $K$ is an object in the closure of $\K$ under kernels of retractions.

Conversely, let $M$ be an object in $\D^b_{dg}(\B')$ such that $M\oplus K'\iso K$ for some objects $K$ and $K'$ in $\K$.
By property (2) of Definition~\ref{def:d-precluster tilting triple}, 
the subcategory $H^0(\I)$ is covariantly finite in $H^0(\B^{(d)}_{\K,\P,\I})$.
Hence the objects in $\K$ are (d+1)-th syzygies of the objects in $\B'$.
We have the following diagram in $\D^b(\B')$
\[
\begin{tikzcd}
M\ar[r]\ar[d,equal]&K\ar[r]\ar[d]&K'\ar[d]\\
M\ar[r]&I\ar[r]\ar[d]&U\ar[d]\\
&V\ar[r,equal]&V
\end{tikzcd}
\]
where $V$ admits an $\I$-coresolution of length $d$. By the Horseshoe Lemma, the object $U$ also admits an $\I$-coresolution of length $d$ and hence $M$ is a (d+1)-th syzygy of the objects in $\B'$.

\end{proof}

\begin{proof}[Proof of Theorem~\ref{thm:Iyama--Solberg correspondence}]
Let $\A$ be a connective $d$-minimal Auslander--Gorenstein exact dg category.  
Let $(\K_{\A},\P_{\A},\I_{\A})$ be the triple formed by the full dg subcategory $\P_{\A}$ on the projectives of $\A$ 
and its full dg subcategory $\I_{\A}$ of projective-injectives 
and the full dg subcategory $\K_{\A}$ of $\D^b_{dg}(\A)$ consisting of the objects $K$  
which are $(d+1)$-st syzygies of objects in $\A$.
By Lemma~\ref{lem:d-minimal AG exact dg category}, the triple $(\K_{\A},\P_{\A},\I_{\A})$ is $d$-precluster tilting.
Therefore the map sending $\A$ to $(\K_{\A},\P_{\A},\I_{\A})$ is well-defined.

Let $(\K,\P,\I)$ be a $d$-precluster tilting triple. %which satisfies properties (1)--(6).
By Proposition~\ref{prop:preclustertilting}, the connective exact dg category $\tau_{\leq 0}\B_{\K,\P,\I}^{(d+1)}$ is $d$-minimal Auslander--Gorenstein. Therefore the map sending $(\K,\P,\I)$ to $\tau_{\leq 0}\B_{\K,\P,\I}^{(d+1)}$ is well-defined.

By Lemma~\ref{lem:d-minimal AG exact dg category} (e), for a connective $d$-minimal Auslander--Gorenstein exact dg category $\A$, we have that $\A$ and $\tau_{\leq 0}\B_{\K_{\A},\P_{\A},\I_{\A}}^{(d+1)}$ are equivalent in the sense of Definition~\ref{def:equ}.

Let $(\K,\P,\I)$ be a $d$-precluster tilting triple. %which satisfies properties (1)--(6).
Put $\B'=\tau_{\leq 0}\B_{\K,\P,\I}^{(d+1)}$.
By Proposition~\ref{prop:preclustertilting} and Proposition~\ref{prop:K}, we have a quasi-equivalence $\pretr(\K)\iso \D^b_{dg}(\B')$ which shows that the $d$-precluster tilting triples $(\K,\P,\I)$ and $(\K_{\B'},\P_{\B'},\I_{\B'})$ are equivalent.
\end{proof}
\begin{example}
{\rm Let $k$ be a field. 
Let $\K$ be an extension-closed subcategory of a pretriangulated dg category $\T_{dg}$.
Then $\K$ inherits a canonical exact structure which we assume to be Frobenius.
Let $\P$ be the full dg subcategory of $\K$ consisting of the projectives and $\I\subset \P$ an additive full dg subcategory which is covariantly finite in $\T=H^0(\T_{dg})$.
We assume that $\P$ is a connective dg category.
Then the triple $(\K,\P,\I)$ is $d$-precluster tilting for any $d\geq 0$.
Indeed, for each object $K$ in $\K$, we have a triangle in $\T$
\[
K'\rightarrow P\rightarrow K\rightarrow \Sigma K'
\]
where $P\in\P$ and $K'\in \K$, since $\K$ is Frobenius.
Then we have  
\[
\Hom_{\T}(Q,\Sigma^{i}K)\iso \Hom_{\T}(Q,\Sigma^{i+1}K')
\]
 for $i\geq 1$ and $Q\in \P$, by the assumption that $\P$ is connective.
 Since $\K$ is Frobenius, we have that $\Hom_{\T}(Q,\Sigma K)=0$ for $Q\in \P$ and $K\in\K$.
 Therefore we have that $\Hom_{\T}(Q,\Sigma^{i}K)=0$ for $i\geq 1$ and $Q\in \P$ and $K\in\K$.
 Similarly we have that $\Hom_{\T}(K,\Sigma^{i}Q)=0$ for $i\geq 1$ and $Q\in \P$ and $K\in\K$.
 
 We have the following two classes of examples.
 The first class of examples is taken from \cite{Wu23a}.
 Let $(Q,F,W)$ be a Jacobi-finite ice quiver with potential. 
Denote by $\Gamma_{rel}$ the relative Ginzburg dg algebra $\Gamma_{rel}(Q,F,W)$, which is a connective dg $k$-algebra.
Let $e=\sum_{i\in F}e_i$ be the idempotent associated with all frozen vertices. 
Let $\pvd_{e}(\Gamma_{rel})$ be the full subcategory of $\pvd(\Gamma_{rel})$ of the dg $\Gamma_{rel}$-modules whose restriction to frozen vertices is acyclic.
Then the {\em relative cluster category} $\C=\C(Q,F,W)$ associated to $(Q,F,W)$ is defined as the Verdier quotient of triangulated categories
\[
\per(\Gamma_{rel})/\pvd_{e}(\Gamma_{rel}).
\]
Let $\pi^{rel}:\per(\Gamma_{rel})\rightarrow \C$ be the quotient functor. 
Since $\Hom_{\per\Gamma_{rel}}(\P,\pvd_{e}(\Gamma_{rel}))=0$, the functor $\pi^{rel}$ restricted to the thick category generated by $\P$ is fully faithful.
By \cite[Theorem 5.42]{Wu23a}, the {\em Higgs category} $\mathcal H$ is equal to the intersection ${^{\perp_{>0}}\P\cap\P^{\perp_{>0}}}$ in $\C$. 
In particular, it is extension-closed in $\C$.
Let $\I\subseteq \P$ be any full additive dg subcategory.
We take canonical dg enhancements of the above categories, where we add the subscripts $dg$ correspondingly.
Then $(\mathcal{H}_{dg},\P_{dg},\I_{dg})$ is $d$-precluster tilting for $d\geq 0$.
 
  The second class is taken from \cite{Jin20}. 
  Let $A$ a dg $k$-algebra. Recall that
\begin{itemize}
\item $A$ is {\em proper} if $\sum_{i\in\mathbb Z}\dim H^i(A)<\infty$; 
\item $A$ is {\em Gorenstein} if the {\em perfect derived category} $\per(A)$ coincides with the thick subcategory of $\D(A)$ generated by $DA$, where $D=\Hom(-,k)$ is the $k$-dual.
\end{itemize}
Let $A$ be a connective proper Gorenstein dg algebra. We denote by $\pvd(A)$ the full subcategory of $\D(A)$ on the perfectly valued dg $A$-modules, i.e.~dg $A$-modules whose total cohomology is finite-dimensional.
A dg $A$-module $M\in\pvd(A)$ is {\em Cohen--Macaulay} \cite[Definition 2.1]{Jin20} if it is connective and if $\Hom_{\D(A)}(M,\Sigma^iA)=0$ for $i>0$.
We denote by $\CM(A)$ the full subcategory of $\pvd(A)$ on the Cohen--Macaulay dg $A$-modules.
Let $I$ be a direct summand of $A$.
Let $\K=\CM_{dg}(A)$ be the canonical dg enhancement of $\CM(A)$ and $\P=\add(A)$ and $\I=\add(I)$ the corresponding full dg subcategories.
Then the triple $(\K,\P,\I)$ is  $d$-precluster tilting for each $d\geq 0$.
Let $\pvd^{\leq 0}(A)$ be the full subcategory of $\pvd(A)$ consisting of the dg $A$-modules whose cohomology is concentrated in degrees $\leq i$.
The $d$-minimal Auslander--Gorenstein exact dg category $\B_{\K,\P,\I}^{(d+1)}$ is a full dg subcategory of $\pvd_{dg}^{\leq 0}(A)$.
When $I=0$, we have that $\pvd_{dg}^{\leq 0}(A)$ is the union of $\B_{\K,\P,\I}^{(d+1)}$ where $d$ runs through non-negative integers, cf.~\cite[Proof of Theorem 3.7]{Jin20}.
%Indeed, for each dg $A$-module $M$, we have $\Hom_{\D(\A)}(A,M)\iso \Hom_{H^0(A)}(H^0(A),H^0(M))\iso H^0(M)$. 
%Then for each connective dg $A$-module $X$ in $\pvd(A)$, we have a morphism $P^0\rightarrow X$ with $P^0\in \P$ which induces an epimorphism $H^0(P^0)\rightarrow H^0(X)$.
%We do this repeatedly and we see that $X$ admits triangles in $\pvd(A)$
%\[
%K=X^{n+1}\rightarrow  P^n\xrightarrow{f^n} X^{n}\rightarrow \Sigma K,
%\]
%\[
%X^{n}\rightarrow P^{n-1}\xrightarrow{f^{n-1}} X^{n-1}\rightarrow \Sigma X^n,
%\]
%\[
%\ldots
%\]
%\[
%X^1\rightarrow P^0\xrightarrow{f^0} X^0=X\rightarrow \Sigma X^1.
%\]
%for $n\geq 0$ and such that $f^i$ induces an epimorphism $H^0(P^i)\rightarrow H^0(X^i)$ for $0\leq i\leq n$. 
%By construction, we have that $K\in \pvd^{\leq 0}(A)$ and 
%\[
%\Hom(K,\Sigma^i A)\iso \Hom(X,\Sigma^{i+n+1} A)\iso \Hom(DA,\Sigma^{i+n+1}DX)
%\]
% for $i\geq 1$. Since $A$ is Gorenstein, it follows that for $n>>0$ we have $K\in \CM(A)$ and $X\in \P\ast\Sigma\P\ast\ldots\ast\Sigma^{n}\P\ast \K$.
}
\end{example}

%Denote by $\Hqe^{ex}$ the (non-full) subcategory of $\Hqe$ consisting of exact dg categories with exact morphisms.

%The above equivalence of $1$-groupoids can be refined to an equivalence of $\infty$-groupoids, cf.~Proposition~\ref{prop:infinitygroupoids}.
\subsection{Refinement to an $\infty$-equivalence}
The correspondence in Theorem~\ref{thm:Iyama--Solberg correspondence} can be refined to an equivalence of $\infty$-groupoids, cf.~Theorem~\ref{thm:infinitygroupoids}.
In this subsection, for a dg $k$-category $\A$, we denote by $\overline{\A}$ the full dg subcategory of $\pretr(\A)$ whose objects are the kernels of retractions of objects in $\A$.
Note that if $\A$ is $d$-minimal Auslander--Gorenstein, then so is $\overline{\A}$. Also, if $(\K,\P,\I)$ is a $d$-precluster tilting triple, then so is $(\overline{\K},\overline{\P},\overline{\I})$.

%Below, we will consider dg functors up to dg isomorphisms. Let us explain this in more detail. Let $\A$ and $\B$ be dg categories and $A$ and $B$ two objects in $\A$. They are {\em dg isomorphic} if they are isomorphic in $Z^0(\A)$. Dg functors $F,G:\A\rightarrow \B$ are {\em dg isomorphic} if they are dg isomorphic in $\Fun_{dg}(\A,\B)$. We have the category $[\dgcat]$ whose objects are small dg categories and whose morphisms are dg functors up to dg isomorphism. A dg functor is a {\em dg equivalence} if it is an isomorphism in $[\dgcat]$. It can also be characterised as follows: a dg functor $F:\A\rightarrow \B$ is a dg equivalence if it is full and faithful and every object of $\B$ is dg isomorphic to an object of $F(\A)$. A dg equivalence is in particular a quasi-equivalence. A dg category $\A$ is {\em strongly pretriangulated} if $\A\rightarrow \pretr(\A)$ is a dg equivalence. In particular, the pretriangulated hull $\pretr(\A)$ is strongly pretriangulated.

Let $\F$ be the category whose objects are $d$-precluster tilting triples $(\K,\P, \I)$, and whose morphisms $(\K,\P,\I)\rightarrow (\K',\P',\I')$ are given by quasi-equivalences $\overline{\K}\rightarrow\overline{\K'}$ 
%(up to dg isomorphisms) 
which restrict to quasi-equivalences $\overline{\P}\rightarrow \overline{\P'}$ and $\overline{\I}\rightarrow \overline{\I'}$.
Let $\G$ be the category whose objects are connective $d$-minimal Auslander--Gorenstein exact dg categories $\A$, and whose morphisms $\A\rightarrow\A'$ are given by 
exact quasi-equivalences $\overline{\A}\rightarrow\overline{\A'}$ of exact dg categories. 

We denote by $N$ the nerve functor $N: \cat\rightarrow \sSet$ from the category $\cat$ of small categories to the category $\sSet$ of simplicial sets, cf.~\cite[1.4]{Cisinski19}. 
It sends a category $\E$ to an $\infty$-category $N(\E)$.
We denote by $h(-):\sSet\rightarrow \cat, X\mapsto h(X)$, its left adjoint. 
An $\infty$-category $\C$ is an $\infty$-groupoid if and only if $h(\C)$ is a $1$-groupoid, cf.~\cite[1.6.7]{Cisinski19}, if and only if $\C$ is a Kan complex, cf.~\cite[3.5.1]{Cisinski19}.
For a subcategory $W$ of an $\infty$-category $\C$, we have the $\infty$-localization~\cite[7.1.2]{Cisinski19} $\C\rightarrow W^{-1}\C$ of $\C$ by $W$.
Then by abstract nonsense, the functor $h(\C)\rightarrow h(W^{-1}\C)$ exhibits $h(W^{-1}\C)$ as the Gabriel--Zisman localization~\cite{GabrielZisman67} of the category $h(\C)$ at the class of morphisms in $h(W)$.
By construction, if $W=\C$ then we have that $h(W^{-1}\C)$ is a groupoid and hence $W^{-1}\C$ is an $\infty$-groupoid.
We denote by $\F'$ the $\infty$-groupoid obtained from $N(\F)$ by the $\infty$-localization at all morphisms.
Similarly, we denote by $\G'$ the $\infty$-groupoid obtained from $N(\G)$ by the $\infty$-localization at all morphisms.
 Let $\tilde{\G}\subseteq \G$ be the full subcategory formed by the objects whose underlying dg category is cofibrant and $\tilde{\G'}$ the $\infty$-localization of $\tilde{\G}$ at all morphisms. 
 We have the cofibrant resolution functor $Q: \G\rightarrow \tilde{\G}, \A\mapsto Q(\A)$. Here we take the exact structure on $Q(\A)$ transferred along the quasi-equivalence $Q(\A)\rightarrow \A$.
Then the inclusion $\inc: \tilde{\G}\rightarrow \G$ induces an equivalence of $\infty$-groupoids $\tilde{\G'}\iso{\G'}$.
Similarly, let $\tilde{\F}\subseteq \F$ be the full subcategory of $d$-precluster tilting triples $(\K,\P,\I)$ such that $\K$ is cofibrant. Let $\tilde{\F'}$ be the $\infty$-localization of $\tilde{\F}$ at all morphisms. 
 We have a functor $Q: \F\rightarrow \tilde{\F}, (\K,\P,\I)\mapsto (Q(\K),\P',\I')$ where $Q(\K)$ is a cofibrant resolution of $\K$ and $\P'$ (resp.~$\I'$) are the essential preimages of $\P$ (resp.~$\I$) under the equivalence $H^0(Q(\K))\iso H^0(\K)$.
Then the inclusion $\inc: \tilde{\F}\rightarrow \F$ induces an equivalence of $\infty$-groupoids $\tilde{\F'}\iso{\F'}$.

We have the truncation functor $\tau_{\leq 0}: \A\mapsto \tau_{\leq 0}\A$ and for each connective dg category $\A$, the canonical dg functor $\tau_{\leq 0}\A\rightarrow \A$ is a quasi-equivalence. 
Therefore, we may assume that the objects in $\G$, $\tilde{\G}$, $\G'$ and $\tilde{\G'}$ are all strictly connective. 

The idea of the proof the following theorem is rather simple: we construct a pair of (ordinary) functors $L:\F\rightarrow \G$ and $R\circ Q:\G\rightarrow \F$ which induces a pair of $\infty$-functors between the $\infty$-groupoids $\F'$ and $\G'$, and then try to show that they are $\infty$-equivalences by showing that there exist natural transformations $R\circ Q\circ L\rightarrow \Id_{\F}$ and $\Id_{\G}\rightarrow L\circ R\circ Q$.
It is not straightforward to carry out this strategy  because in certain constructions, the induced morphisms are not canonical.
\begin{theorem}\label{thm:infinitygroupoids}
We have an equivalence of $\infty$-groupoids $\F'\iso \G'$.
\end{theorem}
\begin{proof}
Let $L(\K,\P,\I)$ be the quasi-essential image of $\tau_{\leq 0}\B_{\K,\P,\I}^{(d+1)}$ under the quasi-equivalence $\pretr(\K)\iso \pretr(\overline{\K})$.
Equivalently, it is the dg subcategory $\tau_{\leq 0}\B_{\overline\K,\overline\P,\overline\I}^{(d+1)}$ of $\pretr(\overline{\K})$.
Clearly, the map $(\K,\P,\I)\mapsto \tau_{\leq 0}\B_{\overline\K,\overline\P,\overline\I}^{(d+1)}$ extends to a functor $L: \F\rightarrow \G$. 
%Indeed, let $\overline{\K}\rightarrow\overline{\K'}$ be a quasi-equivalence which gives a morphism $(\K,\P,\I)\rightarrow (\K',\P',\I')$ in $\F$.
%We have inclusions $\K\hookrightarrow\overline{\K}\hookrightarrow\pretr(\K)\xrightarrow{\sim}\pretr(\overline{\K})$ where the last dg functor is a dg equivalence by \cite[Proposition 4.11]{BondalLarsenLunts04}.
%Since $\B_{\K,\P,\I}^{(d+1)}$ is stable under kernels of retractions in $\pretr(\K)$ (cf.~Remark~\ref{rmk:wic}), we have the following commutative diagram in $[\dgcat]$
%\[
%\begin{tikzcd}
%\K\ar[r,hook]&\overline{\K}\ar[r,hook]&\B_{\K,\P,\I}^{(d+1)}\ar[r,hook]\ar[d]&\pretr(\K)\ar[r,"\sim"]\ar[d]&\pretr(\overline{\K})\ar[d]\\
%\K'\ar[r,hook]&\overline{\K'}\ar[r,hook]&\B_{\K',\P',\I'}^{(d+1)}\ar[r,hook]&\pretr(\K')\ar[r,"\sim"]&\pretr(\overline{\K'})
%\end{tikzcd}.
%\]

Similarly, the map $\A\mapsto (\K_{\overline\A},\P_{\overline\A},\I_{\overline\A})$ extends to a functor $R: \tilde{\G}\rightarrow \F$.
Indeed, let $\overline{\A}\rightarrow\overline{\A'}$ be a quasi-equivalence, which gives rise to a morphism $\A\rightarrow \A'$ in $\tilde{\G}$.
By construction, the bounded dg derived category $\D^b_{dg}(\overline\A)$ is the dg quotient $\pretr(\overline\A)/\N$, 
where $\N$ is the full dg subcategory of $\pretr(\overline\A)$ consisting of the objects in the triangulated subcategory of $\tr(\overline\A)$ generated by the totalizations of conflations in $\overline\A$.
Since both $\A$ and $\A'$ are cofibrant and by Drinfeld's construction of dg quotients \cite{Drinfeld04}, we have a canonical dg functor $\overline\A\rightarrow \D^b_{dg}(\overline\A)$, and any exact dg functor $\overline\A\rightarrow \overline\A'$ induces a canonical dg functor $\D^b_{dg}(\overline\A)\rightarrow \D^b_{dg}(\overline\A')$.
%So we have the following diagram
%\[
%\begin{tikzcd}
%\A\ar[r,hook]&\overline{\A}\ar[r,hook]&\pretr(\A)\ar[r]\ar[d,"\sim"]&\D^b_{dg}(\A)\ar[d,"\sim"]\\
%&&\pretr(\overline{\A})\ar[r]&\D^b_{dg}(\overline{\A})
%\end{tikzcd}.
%\]
So the quasi-equivalence $\overline{\A}\rightarrow \overline{\A'}$ induces a quasi-equivalence $\K_{\overline\A}\rightarrow \K_{\overline\A'}$ which yields a morphism $(\K_{\overline\A},\P_{\overline\A},\I_{\overline\A})\rightarrow (\K_{\overline\A'},\P_{\overline\A'},\I_{\overline\A'})$ in $\F$. 

Let $\A$ be a (strictly) connective exact dg category whose underlying dg category is cofibrant.
Let $\D'$ be the full dg subcategory of $\D^b_{dg}(\overline\A)$ consisting of objects in the quasi-essential image of the dg functor $\overline\A\rightarrow \D^b_{dg}(\overline\A)$.
Then we have the following diagram
\[
\begin{tikzcd}
&&&\tau_{\leq 0}\overline{\D'}\ar[d]&\\
&\tau_{\leq 0}\D'\ar[rru,hook]\ar[r]&\D'\ar[d,hook]\ar[r,hook]&\overline{\D'}\ar[r,hook]&\pretr(\D')\\
\overline\A\ar[r,hook]\ar[ru,"\sim"]&\pretr(\overline\A)\ar[r]&\D^b_{dg}(\overline\A)&
\end{tikzcd}
\]
The map $\A\mapsto \tau_{\leq 0}\overline{\D'}$ extends to a functor $H: \tilde{\G}\rightarrow \G$.
We have a natural transformation $\inc\rightarrow H$ whose value at the object $\A$ is induced by the canonical dg functor $\overline\A\rightarrow \tau_{\leq 0}\overline{\D'}$.

Let $(\K,\P,\I)$ be a $d$-precluster tilting triple with $\K$ being cofibrant.
Let $\V$ be the Drinfeld dg quotient of $\pretr(\overline \K)$ by all the contractible objects. 
Let $\K'$ the full dg subcategory of $\V$ consisting of the objects in $\K$.
Then we have the following commutative diagram in $\dgcat$:
\[
\begin{tikzcd}
\K\ar[r,hook]\ar[d,"\sim"swap]&\overline{\K}\ar[r,hook]&\pretr(\overline\K)\ar[d,"\sim"]\\
\K'\ar[rr,hook]&&\V
\end{tikzcd}
\]
The map $\K\mapsto \K'$ extends to a functor $I: \tilde{\F}\rightarrow \F$ and we have a natural transformation $\inc\rightarrow I$ whose value at the object $(\K,\P,\I)$ is induced by the canonical quasi-equivalence $\K\rightarrow \K'$. 

We construct a natural transformation $R \circ Q\circ L\circ \inc \rightarrow I$ as follows: consider a $d$-precluster tilting triple $(\K,\P,\I)$ with $\K$ being cofibrant. Put $\B=\B_{\overline\K,\overline\P,\overline\I}^{(d+1)}$ and $\B'=\tau_{\leq 0}\B$. 
The dg functor $Q(\B')\rightarrow \B'\rightarrow \pretr(\overline\K)$ induces a canonical dg functor  $\pretr(Q(\B'))\rightarrow \pretr(\overline\K)$.
When restricted to the full dg subcategory $\overline{Q(\B')}$, it extends to a canonical dg functor $\pretr(\overline{Q(\B')})\rightarrow \pretr(\overline{\K})$.
%For each contractible object $N$ in $\pretr(\K)$, we fix a homotopy $h_{N}$.
By the construction of $\V$, we obtain a canonical dg functor $\D^b_{dg}(\overline{Q(\B')})\rightarrow \V$ extending the dg functor $\pretr(\overline{Q(\B')})\rightarrow \pretr(\overline\K)\rightarrow \V$.
By Lemma~\ref{lem:pretri}, it is a quasi-equivalence.
Consequently, we have the following commutative diagram in $\dgcat$
\[
\begin{tikzcd}
&Q(\B')\ar[r]\ar[d]&\B'=\tau_{\leq 0}\B\ar[d]\\
\K_{\overline{Q(\B')}}\ar[r,hook]\ar[dd,bend right=6ex,dashed]\ar[rd,"\sim"swap]&\D^b_{dg}(\overline{Q(\B')})\ar[rdd,dashed]&\B\ar[d,hook]\\
\K\ar[r,hook]\ar[d]&\overline{\K}\ar[ru,hook]&\pretr(\overline\K)\ar[d,"\sim"]\\
\K'\ar[rr,hook]&&\V
\end{tikzcd}.
\]
The quasi-equivalence $\K_{\overline{\Q(\B')}}\rightarrow \K' $ yields the desired morphism $R\circ Q\circ L\circ \inc(\K,\P,\I)\rightarrow I(\K,\P,\I)$ in $\F$.

We also establish a natural transformation $L \circ R \rightarrow H$ as follows: consider an object $\A$ in $\tilde{\G}$.
Put $\B=\B_{\overline{\K_{\overline\A}},\overline{\P_{\overline\A}},\overline{\I_{\overline\A}}}^{(d+1)}$ and $\B'=\tau_{\leq 0}\B$.
 Then, we have the following diagram
\[
\begin{tikzcd}
&&&\B'\ar[d]\ar[lld,bend right=6ex,dashed]&\\
&{\tau_{\leq 0}{\overline{\D'}}}\ar[rrd]&\K_{\overline\A}\ar[d,hook]\ar[r]&\B\ar[r]\ar[d,dashed]&\pretr(\overline{\K_{\overline\A}})\ar[d,dashed]\\
&\tau_{\leq 0}\D'\ar[r]\ar[u,hook]&\D'\ar[d,hook]\ar[r,hook]&\overline{\D'}\ar[r,hook]&\pretr(\D')\\
\overline\A\ar[r,hook]\ar[ru,"\sim"]&\pretr(\overline\A)\ar[r]&\D^b_{dg}(\overline\A)&&
\end{tikzcd}
\]
On the object $\A$, the natural transformation $L \circ R \rightarrow H$ is induced by the quasi-equivalence $\B'\rightarrow \tau_{\leq 0}\overline{\D'}$.

Summarizing the above, we have the following diagram of functors
\begin{equation}
\begin{tikzcd}\label{dia:func}
\tilde{\F}\ar[rr,"\inc"{description},shift left=2ex]\ar[rr,"I"{description}]&&\F\ar[ll,shift right=4ex,"Q"{description}]\ar[d,"L"{description}]\\
\tilde{\G}\ar[rr,"\inc"{description},shift right=2ex]\ar[rr,"H"{description},shift right =4ex]\ar[rru,"R"{description}]&&\G\ar[ll,"Q"{description}]
\end{tikzcd}
\end{equation}
where we have natural transformations $R\circ Q\circ L\circ \inc\rightarrow I$ and $L\circ R\rightarrow H$ which are objectwise quasi-equivalences.
By the universal property of $\infty$-localizations, we have a diagram of $\infty$-groupoids corresponding to the diagram (\ref{dia:func}) and we keep the notations for the $\infty$-functors between the $\infty$-groupoids. Then we have 
\[
R\circ Q\circ L\iso R\circ Q\circ L\circ \inc\circ Q\iso I\circ Q\iso \inc \circ Q\iso \Id_{\F'}
\]
and 
\[
L\circ R\circ Q\iso H\circ Q\iso \inc\circ Q\iso \Id_{\G'}.
\]
Hence the functor $L:\F\rightarrow \G$ induces an $\infty$-equivalence of $\infty$-groupoids $\F'\rightarrow \G'$.
\end{proof}

\begin{remark}
It is clear that the full subcategory of $\F$ on the objects of the form $(\K,0,0)$ (resp.~$(\P,\P,\P)$, $(\P,\P,0)$) constitutes a connected component of the category $\F$.
By~\cite[Corollary 8.4]{Toen07}, the homotopy groups are given as follows 
\[
\pi_{i}(\F', (\K,0,0))\iso HH^{2-i}(\overline{\K}), \;\;\; \forall i>2,
\] 
\[
\pi_{2}(\F', (\K,0,0))\iso HH^{0}(\overline{\K})^{*};
\]
\[
\pi_{i}(\F', (\P,\P,\P))\iso HH^{2-i}(\overline{\P}), \;\;\; \forall i>2,
\] 
\[
\pi_{2}(\F', (\P,\P,\P))\iso HH^{0}(\overline{\P})^{*};
\]
\[
\pi_{i}(\F', (\P,\P,0))\iso HH^{2-i}(\overline{\P}), \;\;\; \forall i>2,
\] 
\[
\pi_{2}(\F', (\P,\P,0))\iso HH^{0}(\overline{\P})^{*},
\]
 where for a small dg category $\A$, $HH^0(\A)^*$ denotes the automorphism group of the dg $\A$-$\A$-bimodule $\A$ in $\D(\A^{e})$.
 It would be interesting to compute the homotopy groups $\pi_i(\F', (\K,\P,\I))$ for general $d$-precluster tilting triple $(\K,\P,\I)$.
\end{remark}

\section{Connection to exact categories}
In this section we give necessary and sufficient conditions on a $d$-precluster tilting triple $(\K,\P,\I)$ for the connective $d$-minimal Aulander--Gorenstein exact dg category $\tau_{\leq 0}\B_{\K,\P,\I}^{(d+1)}$ to be concentrated in degree zero (resp.~to be an abelian category).
In the abelian case, we will see that $\K$ and $\I$ are determined by $\P$ and that there is a fully faithful functor $\P\rightarrow (\I\mbox{-}\mod)^{op}$.
Assume $k$ is a commutative artinian ring.
 When $\tau_{\leq 0}\B_{\K,\P,\I}^{(d+1)}$ is the module category over an Artin $k$-algebra $\Gamma$, we may identify $\I$ with the full subcategory of injective modules over some Artin $k$-algebra $\Lambda$ and identify $\P$ with an additive full subcategory of $\mod\Lambda$.
 Suppose that $\P=\add M$ for some $\Lambda$-module $M$. 
We show that the thus defined triple $(\K,\P,\I)$ is $d$-precluster tilting if and only if $M$ is a $d$-precluster tilting $\Lambda$-module.
This justifies our terminology to some extent.
\subsection{Exact case}
\begin{proposition}\label{prop:Quillenexact}
Let $(\K,\P,\I)$ be a $d$-precluster tilting triple. %satisfying properties (1)--(5). 
The connective exact dg category $\tau_{\leq 0}\B_{\K,\P,\I}^{(d+1)}$ is concentrated in degree zero if and only if the following conditions hold
\begin{itemize}
\item[(a)] The dg category $\K$ is concentrated in non-negative degrees (in particular, the dg category $\P$ is concentrated in degree zero).  
\item[(b)] Consider a complex in $H^0(\K)$
\begin{equation}\label{cplx:HK}
\begin{tikzcd}
0\ar[r]&K_{d+1}\ar[r]& P_d\ar[r]&P_{d-1}\ar[r]&\ldots\ar[r]&P_1\ar[r]&P_0
\end{tikzcd}
\end{equation}
where $K_{d+1}\in H^0(\K)$ and $P_i\in H^0(\P)$ for $0\leq i\leq d$.
If the induced complex of Hom spaces in $H^0(\K)$
\begin{equation}\label{cplx:hom}
\begin{tikzcd}
(P_0,I)\ar[r]&(P_1,I)\ar[r]&\ldots\ar[r]&(P_{d-1},I)\ar[r]&(P_d,I)\ar[r]&(K_{d+1},I)\ar[r]&0
\end{tikzcd}
\end{equation}
is exact for each $I\in\I$, then the corresponding complex of Hom spaces in $H^0(\K)$
\begin{equation}\label{cplx:K}
\begin{tikzcd}
0\ar[r]&(P,K_{d+1})\ar[r]& (P,P_d)\ar[r]&(P,P_{d-1})\ar[r]&\ldots\ar[r]&(P,P_1)\ar[r]&(P,P_0)
\end{tikzcd}
\end{equation}
is exact for each $P\in\P$.
\end{itemize}
\end{proposition}
\begin{proof}
Let $(\K,\P,\I)$ be a $d$-precluster tilting triple satisfying conditions (a) and (b). Put $\B=\B_{\K,\P,\I}^{(d+1)}$ and $\B'=\tau_{\leq 0}\B$.
Let $X$ be an object in $\B$.
We have conflations  in $H^0(\B')$
\begin{equation}\label{conf:X0}
\begin{tikzcd}[row sep=small]
X_1\ar[r]& P_0\ar[r]& X=X_0\ar[r,dashed,"\delta_0"]&;\\
X_2\ar[r]& P_1\ar[r]& X_1\ar[r,dashed,"\delta_1"]&;\\
&\ldots&&\\
X_{d+1}=K_{d+1}\ar[r]& P_d\ar[r]& X_d\ar[r,dashed,"\delta_d"]&,
\end{tikzcd}
\end{equation}
 where $K_{d+1}\in\K$ and $P_i\in\P$ for $0\leq i\leq d$. 
 By condition (a), we have $\Hom_{\tr(\K)}(\K,\Sigma^{-i}\P)=0$ for $i\geq 1$ and therefore $\Hom_{\tr(\K)}(X_j,\Sigma^{-i}\P)=0$ for $i\geq 1$ and $0\leq j\leq d$. 
 It follows that the complex (\ref{cplx:hom}) is exact for each $I\in\I$ and hence by condition (b), the complex (\ref{cplx:K}) is exact for $P\in\P$.
 Since objects in $\P$ are projective in $H^0(\B')$, the maps $P_i\rightarrow X_i$ induce surjections $\Hom(P,P_i)\rightarrow \Hom(P,X_i)$ for $P\in\P$ and $0\leq i\leq d$.
 Therefore the maps $\Hom(P,X_{i+1})\rightarrow \Hom(P,P_i)$ are injections for $0\leq i\leq d$ and $P\in\P$.
 It follows that we have $\Hom(P,\Sigma^{-i}X_j)=0$ for $P\in\P$ and $i\geq 1$ and $0\leq j\leq d+1$.
 Let $X'$ be another object in $\B'$ with conflations given by adding a prime symbol to each object in the conflations (\ref{conf:X0}).
 Then we have 
 \[
 \Hom(X',\Sigma^{-1}X)\iso \Hom(X_1',\Sigma^{-2}X)\iso\ldots\iso\Hom(K_{d+1},\Sigma^{-d-2}X)
 \]
and 
\[
\Hom(K,\Sigma^{-d-2}X)\iso\Hom(K,\Sigma^{-d-1}X_1)\iso\ldots\iso\Hom(K,\Sigma^{-1}K_{d+1})=0
\]
for $K\in\K$.
Similarly we have $\Hom(X',\Sigma^{-i}X)=0$ for $i\geq 1$.
Therefore the connective exact dg category $\B'$ is concentrated in degree zero.

Conversely, we suppose that $\B'$ is concentrated in degree zero. 
Since $\tau_{\leq 0}\K$ is a full dg subcategory of $\B'$, we have that $\tau_{\leq 0}\K$ is concentrated in degree zero and hence $\K$ is concentrated in non-negative degrees. 
This shows condition (a).

Let us show condition (b). 
Since $\B'$ is concentrated in degree zero, a complex (\ref{cplx:HK}) in $H^0(\K)$ such that the corresponding complex (\ref{cplx:hom}) is exact for each $I\in\I$, where $K_{d+1}\in H^0(\K)$ and $P_i\in H^0(\P)$ for $0\leq i\leq d$, is equivalent to a sequence of conflations (\ref{conf:X0}) in $H^0(\B')$.
Let $P$ be an object in $\P$. If we apply the functor $\Hom(P,-)$ to the sequences in (\ref{conf:X0}), they become short exact sequences of Hom spaces in $H^0(\K)$.
Now the exactness of the complex (\ref{cplx:K}) follows by pasting the above short exact sequences.
\end{proof}
\begin{lemma}\label{lem:Quillenexact}
Let $(\K,\P,\I)$ be a $d$-precluster tilting triple which satisfies %properties (1)--(5) and 
the conditions (a) and (b) in Proposition~\ref{prop:Quillenexact}. 
Then the map sending $X\in \B'=\tau_{\leq 0}\B_{\K,\P,\I}^{(d+1)}$ to $\Hom_{\B'}(-,X)|_{\P}$ 
extends to a fully faithful exact functor $F:\B'\hookrightarrow \Mod\P$.
It induces a bijection 
\[
\Ext^1_{\B'}(X,X')\iso \Ext^1_{\Mod\P}(FX,FX')
\]
 for $X$ and $X'$ in $\B'$.
\end{lemma}
\begin{proof}
By Proposition~\ref{prop:Quillenexact}, the exact dg category $\B'$ is concentrated in degree zero and hence is an exact category.
Since the objects in $\P$ are projective in $\B'$, the functor $F$ is an exact functor.

Let $X$ be an object in $\B'$ with conflations (\ref{conf:X0}).
Recall that for an object $P\in\P$, we write $P^{\wedge}$ for the representable right $\P$-module $\Hom_{\P}(-,P)$.
The right $\P$-module $\Hom_{\B'}(-,X)|_{\P}$ is the cokernel of the morphism $P_1^{\wedge}\rightarrow P_0^{\wedge}$.
%Let $P$ be any object in $\P$. 
%Then we have the following diagram 
%\[
%\begin{tikzcd}
%\Hom_{\P}(P,P_1)\ar[r]\ar[d,"\simeq"swap]&\Hom_{\P}(P,P_0)\ar[r]\ar[d,"\simeq"swap]&\Hom_{\B'}(P,X)\ar[r]\ar[d,dashed]&0\\
%\Hom_{\Mod\P}(P^{\wedge},P_1^{\wedge})\ar[r]&\Hom_{\Mod\P}(P^{\wedge},P_0^{\wedge})\ar[r]&\Hom_{\Mod\P}(P^{\wedge},FX)\ar[r]& 0
%\end{tikzcd}.
%\]
%Hence we have $\Hom_{\B'}(P,X)\iso\Hom_{\Mod\P}(FP,FX)$. 
%Similarly we have $\Hom_{\B'}(X,P)\iso\Hom_{\Mod\P}(FX,FP)$.
Let $X'$ be another object in $\B'$.
 Then we have the following diagram
\[
\begin{tikzcd}%[column sep=small]
0\ar[r]&\Hom_{\B'}(X,X')\ar[r]\ar[d,dashed,""swap]&\Hom_{\B'}(P_0,X')\ar[r]\ar[d,"\simeq"swap]&\Hom_{\B'}(P_1,X')\ar[d,"\simeq"]\ar[r]&\Hom_{\B'}(P_2,X')\ar[d,"\simeq"swap]\\
0\ar[r]&(FX,FX')\ar[r]&(FP_0,FX')\ar[r]&(FP_1,FX')\ar[r]&(FP_2,FX')
\end{tikzcd}.
\]
where the second row is a sequence of Hom spaces in $\Mod\P$. 
Hence we have $\Hom_{\B'}(X,X')\iso\Hom_{\Mod\P}(FX,FX')$ and the functor $F$ is fully faithful. 
Since objects in $\P$ (resp.~$F(\P)$) are projective in $\B'$ (resp.~$\Mod\P$), we have a natural bijection
\[
\Ext^1_{\B'}(X,X')\iso \Ext^1_{\Mod\P}(FX,FX').
\]
\end{proof}

\subsection{Abelian case}
\begin{proposition}\label{prop:Abelian}
Let $(\K,\P,\I)$ be a $d$-precluster tilting triple. 
The connective exact dg category $\B'=\tau_{\leq 0}\B_{\K,\P,\I}^{(d+1)}$ is abelian if and only if the following conditions hold
\begin{itemize}
\item[(a)] The dg category $\K$ is concentrated in non-negative degrees (in particular, the dg category $\P$ is concentrated in degree zero) and the category $\P$ admits weak kernels.  
\item[(b)] Consider a complex in $H^0(\K)$
\begin{equation}\label{cplx:HK2}
\begin{tikzcd}
0\ar[r]&K_{d+1}\ar[r]& P_d\ar[r]&P_{d-1}\ar[r]&\ldots\ar[r]&P_1\ar[r]&P_0
\end{tikzcd}
\end{equation}
where $K_{d+1}\in H^0(\K)$ and $P_i\in H^0(\P)$ for $0\leq i\leq d$.
 The induced complex of Hom spaces in $H^0(\K)$
\begin{equation}\label{cplx:hom2}
\begin{tikzcd}
(P_0,I)\ar[r]&(P_1,I)\ar[r]&\ldots\ar[r]&(P_{d-1},I)\ar[r]&(P_d,I)\ar[r]&(K_{d+1},I)\ar[r]&0
\end{tikzcd}
\end{equation}
is exact for each $I\in\I$, if and only if the corresponding complex of Hom spaces in $H^0(\K)$
\begin{equation}\label{cplx:K2}
\begin{tikzcd}
0\ar[r]&(P,K_{d+1})\ar[r]& (P,P_d)\ar[r]&(P,P_{d-1})\ar[r]&\ldots\ar[r]&(P,P_1)\ar[r]&(P,P_0)
\end{tikzcd}
\end{equation}
is exact for each $P\in\P$.
\item[(c)] Each morphism $P_1\rightarrow P_0$ in $H^0(\P)$ can be completed to a complex of the form (\ref{cplx:HK2}) where $K_{d+1}\in H^0(\K)$ and $P_i\in H^0(\P)$ for $0\leq i\leq d$ such that the corresponding complex (\ref{cplx:K2}) is exact for each $P\in\P$. In other words, the functor $F:\B'\hookrightarrow \mod\P$ defined in Lemma~\ref{lem:Quillenexact} is dense.%Each object in $\mod\P$ is isomorphic to $\cok(P_1^{\wedge}\rightarrow P_0^{\wedge})$ for some morphism $P_1\rightarrow P_0$ in $\P$ that can be completed to a complex (\ref{cplx:HK2}), where $K_{d+1}\in H^0(\K)$ and $P_i\in H^0(\P)$ for $0\leq i\leq d$, such that the complex (\ref{cplx:hom2}) is exact for each $I\in\I$.

\end{itemize}
\end{proposition}
\begin{proof}
Let $(\K,\P,\I)$ be a $d$-precluster tilting triple which satisfies conditions (a), (b) and (c) above.
Let us show that $\B'$ is abelian.
By Lemma~\ref{lem:Quillenexact}, we have an exact embedding $F:\B'\hookrightarrow \mod \P$ which induces a bijection between the $\Ext^1(-,-)$.
By condition (c), the functor $F$ is dense and hence $F$ is an equivalence of categories.
%Since both categories have enough projectives and $F$ preserves projectives, the functor $F$ is an equivalence of exact categories.
It follows that $\B'\iso\mod\P$ is an abelian category.

Conversely we assume that $\B'$ (with the induced exact structure from $\tr(\K)$) is an abelian category. 
Since $\B'$ has enough projectives, it follows that $\P$ has weak kernels.
This proves condition (a).
Let us prove the ``if" part of condition (b). 
For each object $Y$ in $\B'$, there is a conflation
\[
\begin{tikzcd}
Y_1\ar[r,tail]& Q_0\ar[r,two heads]&Y
\end{tikzcd}
\]
where $Q_0\in\P$. Then the morphism $Q_0\rightarrow Y$ is an epimorphism. Since the morphism $(P,K_{d+1})\rightarrowtail (P,P_d)$ is an injection for each $P\in\P$, the morphism $K_{d+1}\rightarrow P_d$ is a monomorphism in $\B'$.
Therefore we have a triangle in $\tr(\K)$
\[
\begin{tikzcd}
K_{d+1}\ar[r]& P_d\ar[r]&X_d\ar[r]&\Sigma K_{d+1}
\end{tikzcd}
\]
where $X_d\in \B'$. Since $\Hom_{\tr(\K)}(\Sigma K_{d+1}, P_{d-1})=0$, we have a unique map $X_d\rightarrow P_{d-1}$ such that the following diagram commutes
\[
\begin{tikzcd}
P_{d}\ar[d]\ar[r]&P_{d-1}\\
X_{d}\ar[ru,dashed,"\exists!"swap]&
\end{tikzcd}.
\]
A similar argument shows that the morphism $X_d\rightarrow P_{d-1}$ is a monomorphism in $\B'$.
We repeat the above procedures and we see that we have a sequence of conflations (\ref{conf:X0}) in $H^0(\B')$.
By definition, we see that the complex \ref{cplx:hom2} is exact for each $I\in\I$.

By Lemma~\ref{lem:Quillenexact}, we have an exact embedding $F:\B'\hookrightarrow \mod \P$ and the representable modules $P^{\wedge}$ are in the essential image.
Therefore the functor $F$ is an equivalence of abelian categories.
This shows condition (c).
%Suppose we have a morphism $P_1'\rightarrow P_0'$ in $H^0(\P)$.
%Then we have some $X\in \B'$ such that 
%\[
%F(X)\iso\cok({P_1'}^{\wedge}\rightarrow {P_0'}^{\wedge}).
%\]
%So we have the following diagram
%\[
%\begin{tikzcd}
%P_1^{\wedge}\ar[rd,two heads]&&P_0^{\wedge}\ar[r]\ar[dd]&F(X)\ar[dd,equal]\ar[r,two heads]&0\\
%&U\ar[ru,tail]\ar[dd]&&&\\
%{P_1'}^{\wedge}\ar[rd,two heads]&&{P_0'}^{\wedge}\ar[r]&F(X)\ar[r,two heads]&0\\
%&V\ar[ru,tail]&&&
%\end{tikzcd}.
%\]
%It follows that $U\oplus {P_0'}^{\wedge}\iso V\oplus P_0^{\wedge}$.
\end{proof}
\begin{corollary}\label{cor:abelian}
Let $(\K,\P,\I)$ be a $d$-precluster tilting triple %which satisfying properties (1)--(5) 
such that $\B'=\tau_{\leq 0}\B_{\K,\P,\I}^{(d+1)}$ is abelian.
Then the dg category $\P$ is concentrated in degree zero with weak kernels and the map sending $X\in \B'$ to $\Hom_{\B'}(-,X)|_{\P}$ 
extends to an exact equivalence $F:\B'\iso \mod\P$.
\end{corollary}
\begin{lemma}\label{lem:injectives}
Let $(\K,\P,\I)$ be a $d$-precluster tilting triple %which satisfying properties (1)--(5) 
such that $\B'=\tau_{\leq 0}\B_{\K,\P,\I}^{(d+1)}$ is abelian.
Then the map sending an object $X$ to $\Hom_{\B'}(X,-)|_{\I}$ extends to a 
functor $G:\B'\rightarrow (\I\mbox{-}\mod)^{op}$, which is fully faithful when restricted to the full subcategory $H^0(\K)$.
\end{lemma}
\begin{proof}
Let $K$ and $K'$ be objects in $\K$. 
Since $H^0(\I)$ is covariantly finite in $H^0(\B_{\K,\P,\I}^{(d)})$, we have triangles in $\tr(\K)$
\[
K\rightarrow I_0\rightarrow U\rightarrow \Sigma K,
\]
\[
U\rightarrow I_1\rightarrow V\rightarrow \Sigma U,
\]
where $I_0$ and $I_1$ are in $\I$ and $U$ and $V$ are in $\B'$.
Then we have short exact sequences of left $\I$-modules
\[
0\rightarrow \Hom(U,-)|_{\I}\rightarrow \Hom(I_0,-)|_{\I}\rightarrow \Hom(K,-)|_{\I}\rightarrow 0,
\]
\[
0\rightarrow \Hom(V,-)|_{\I}\rightarrow \Hom(I_1,-)|_{\I}\rightarrow \Hom(U,-)|_{\I}\rightarrow 0.
\]
So we have an exact sequence in $(\I\mbox{-}\mod)^{op}$
\[
0\rightarrow G(K)\rightarrow G(I_0)\rightarrow G(I_1)
\]
We also have exact sequences of Hom spaces in $\tr(\K)$
\[
0=\Hom(K',\Sigma^{-1}U)\rightarrow \Hom(K',K)\rightarrow \Hom(K',I_0)\rightarrow \Hom(K',U), 
\]
\[
0=\Hom(K',\Sigma^{-1}V)\rightarrow \Hom(K',U)\rightarrow \Hom(K',I_1)\rightarrow \Hom(K',V).
\]
So we have the following diagram
\[
\begin{tikzcd}
0\ar[r]&\Hom_{\B'}(K',K)\ar[r]\ar[d,dashed,""swap]&\Hom_{\B'}(K',I_0)\ar[r]\ar[d,"\simeq"swap]&\Hom_{\B'}(K',I_1)\ar[d,"\simeq"]\\
0\ar[r]&(GK',GK)\ar[r]&(GK',GI_0)\ar[r]&(GK',GI_1)
\end{tikzcd}
\]
where the second row is a sequence of Hom spaces in $(\I\mbox{-}\mod)^{op}$.
So the functor $G|_{H^0(\K)}$ is fully faithful.

\end{proof}
\subsection{Iyama--Solberg's correspondence}\label{subsec:IS}
In this subsection, we assume that $k$ is a commutative artin ring. We compare our correspondence with the correspondence obtained in \cite[Theorem 4.5]{IyamaSolberg18}. 
Let $\Lambda$ be an Artin $k$-algebra and $\I$ the full subcategory of injective $\Lambda$-modules.

In \cite[1.4.1]{Iyama07a}, for $d\geq 1$, Iyama defined the  functors
\[
\tau_{d}=\tau\Omega^{d-1}:\underline{\mod}\Lambda\rightarrow\overline{\mod}\Lambda\text{ and } \tau_{d}^{-}=\tau^{-}\Omega^{-(d-1)}: \overline{\mod}\Lambda\rightarrow\underline{\mod}\Lambda
\]
as the {\em $d$-Auslander--Reiten translation}.

\begin{definition}[{\cite[Definition 3.2]{IyamaSolberg18}}]\label{def:d-precluster tilting}
A subcategory $\C$ of $\mod\Lambda$ is a {\em $d$-precluster tilting subcategory} if it satisfies the following conditions.
\begin{itemize}
\item[(i)] $\C$ is a generator-cogenerator for $\mod\Lambda$.
\item[(ii)] $\tau_{d}(\C)\subseteq \C$ and $\tau_{d}^{-}(\C)\subseteq \C$.
\item[(iii)] $\Ext^{i}_{\Lambda}(\C,\C)=0$ for $0<i<d$.
\item[(iv)] $\C$ is a functorially finite subcategory of $\mod\Lambda$.
\end{itemize}
If moreover $\C$ admits an additive generator $M$, we say that $M$ is a {\em $d$-precluster tilting module}.
\end{definition}
Put 
\begin{align*}
{^{\perp_{d-1}}\C}=\{\text{$X$ in $\mod\Lambda\mid\Ext^{i}_{\Lambda}(X,\C)=0$ for $0<i<d$}\},\\
{\C^{\perp_{d-1}}}=\{\text{$X$ in $\mod\Lambda\mid\Ext^{i}_{\Lambda}(\C,X)=0$ for $0<i<d$}\}.
\end{align*}
Suppose conditions (i), (iii) and (iv) hold and $d>1$, then by \cite[Proposition 3.8 (b)]{IyamaSolberg18}, condition (ii) is equivalent to ${^{\perp_{d-1}}\C}={\C^{\perp_{d-1}}}$.
Indeed, we have the following result. 
Here we include a proof for completeness.
\begin{lemma}\label{lem:taud}%[{\cite[Proposition 3.8]{IyamaSolberg18}}]
Suppose $d>1$. Let $\C$ be an additive subcategory of $\mod\Lambda$ satisfying  the conditions (i), (iii) and (iv) in Definition~\ref{def:d-precluster tilting}. 
Then $\tau_{d}(\C)\subseteq \C$ (resp.~$\tau_{d}^{-}(\C)\subseteq \C$) if and only if ${^{\perp_{d-1}}\C}\subseteq \C^{\perp_{d-1}}$ (resp.~${^{\perp_{d-1}}\C}\supseteq \C^{\perp_{d-1}}$).
\end{lemma}
\begin{proof}
We prove the case for $\tau_d$ and the case for $\tau_{d}^{-}$ is similar.

Assume $\tau_{d}(\C)\subseteq \C$.
Since $\Lambda$ is contained in $\C$, we have $\C\subseteq {^{\perp_{d-1}}\Lambda}$ by condition (iii). 
So we have 
\[
\Ext^i_{\Lambda}(\C,{^{\perp_{d-1}}\C})\iso D\Ext^{d-i}_{\Lambda}({^{\perp_{d-1}}}\C,\tau_d(\C))=0
\] 
for $0<i<d$. 
It follows that ${^{\perp_{d-1}}}\C\subseteq \C^{\perp_{d-1}}$. 
%Notice that in the above proof we only used conditions (i) and (iii).

Conversely, we assume that ${^{\perp_{d-1}}}\C\subseteq \C^{\perp_{d-1}}$.
Since $D\Lambda$ is contained in $\C$, it follows that $\C^{\perp_{d-1}}\subseteq (D\Lambda)^{\perp_{d-1}}$.
Hence ${^{\perp_{d-1}}}\C\subseteq (D\Lambda)^{\perp_{d-1}}$ and we have 
\[
\Ext^i_{\Lambda}(\C,{^{\perp_{d-1}}}\C)\iso D\Ext^{d-i}_{\Lambda}(\tau_{d}^{-}({^{\perp_{d-1}}}\C),\C)=0
\]
for $0<i<d$.
Hence $\tau_{d}^{-}({^{\perp_{d-1}}}\C)\subseteq {^{\perp_{d-1}}}\C$.
Since ${^{\perp_{d-1}}}\C\subseteq (D\Lambda)^{\perp_{d-1}}$, we have
\[
\Ext^i_{\Lambda}(({^{\perp_{d-1}}}\C)^{\perp_{d-1}},{^{\perp_{d-1}}}\C)\iso D\Ext^{d-i}_{\Lambda}(\tau_{d}^{-}({^{\perp_{d-1}}}\C),({^{\perp_{d-1}}}\C)^{\perp_{d-1}})=0
\]
for $0<i<d$. 
It follows that $({^{\perp_{d-1}}}\C)^{\perp_{d-1}}\subseteq {^{\perp_{d-1}}({^{\perp_{d-1}}}\C)}\subseteq {^{\perp_{d-1}}}\C$.
We show that $({^{\perp_{d-1}}}\C)^{\perp_{d-1}}=\C$.
Let $X$ be an object in $({^{\perp_{d-1}}}\C)^{\perp_{d-1}}$.
By conditions (i) and (iv), we have an exact sequence 
\begin{equation}\label{seq:Xexact}
\begin{tikzcd}
0\ar[r]&X\ar[r,tail,"f"]&C\ar[r, two heads]&Y\ar[r]&0
\end{tikzcd}
\end{equation}
where $f$ is a left $\C$-approximation of $X$.
Then $Y\in {^{\perp_{d-1}}}\C$ and the sequence (\ref{seq:Xexact}) splits.  
Hence we have $X\in \C$ and $({^{\perp_{d-1}}}\C)^{\perp_{d-1}}=\C$.
Since $\C\subseteq {^{\perp_{d-1}}\Lambda}$, by the assumption that ${^{\perp_{d-1}}}\C\subseteq \C^{\perp_{d-1}}$ we have 
\[
\Ext^i_{\Lambda}({^{\perp_{d-1}}}\C,\tau_d(\C))\iso D\Ext^{d-i}_{\Lambda}(\C,{^{\perp_{d-1}}}\C)=0
\]
for $0<i<d$. It follows that $\tau_{d}(\C)\subseteq ({^{\perp_{d-1}}}\C)^{\perp_{d-1}}=\C$ and we are done.
%Since $D\Lambda$ is contained in $\C$, it is contained in ${^{\perp_{d-1}}}\C$.
\end{proof}
Let $(\K,\P,\I)$ be a $d$-precluster tilting triple which satisfies %satisfying properties (1)--(6) and 
conditions (a), (b) and (c) in Proposition~\ref{prop:Abelian}.
% and $\Gamma=\Hom_{\Lambda}(M,M)^{op}$.
By Corollary~\ref{cor:abelian}, the connective exact dg category 
$\tau_{\leq 0}\B_{\K,\P,\I}^{(d+1)}$ can be identified with the abelian category $\mod\P$.
Then the category $\I$ can be identified with the full subcategory of projective-injectives in $\mod\P$, 
and the dg category $\K$ can be identified with the full dg subcategory 
of $\D^b_{dg}(\mod\P)$ consisting of all $(d+1)$-st syzygies of $\P$-modules. 
So $\K$ and $\I$ are determined by $\P$ if $\P$ is part of a $d$-precluster tilting triple $(\K,\P,\I)$ 
such that $\B'=\tau_{\leq 0}\B_{\K,\P,\I}^{(d+1)}$ is abelian.

Now, suppose that $\P=\add M$ for some $\Lambda$-module $M$ where $\Lambda$ is an Artin algebra, and by Lemma~\ref{lem:injectives}, we also assume that $\I=\inj\Lambda$.
We identify $\P$, respectively $\I$, with their essential images under the Yoneda embedding $\P\hookrightarrow \mod\P$. 
We further identify them with their essential images under the dg functor $\mod\P\rightarrow \D^b_{dg}(\mod\P)$ and view them as full dg subcategories of $\K$, which is a full dg subcategory of $\D^b_{dg}(\mod\P)$ consisting of all $(d+1)$-st syzygies of $\P$-modules.
We claim that the triple $(\K,\P,\I)$ is $d$-precluster tilting satisfying the conditions in Proposition~\ref{prop:Abelian} if and only if  the $\Lambda$-module $M$ is $d$-precluster tilting.
We identify $\pretr(\K)$ with $\D^b_{dg}(\mod\P)$ via the canonical quasi-equivalence $\pretr(\K)\iso \D^b_{dg}(\mod\P)$ induced by the inclusion $\K\hookrightarrow \D^b_{dg}(\mod\P)$.

\begin{lemma}
Assume that $M$ is a $d$-precluster tilting $\Lambda$-module. 
Then the triple $(\K,\P,\I)$ defined above is $d$-precluster tilting and satisfies %properties (1)--(6) and 
conditions (a), (b) and (c) in Proposition~\ref{prop:Abelian}.
\end{lemma}
\begin{proof}
Property (1) in Definition~\ref{def:d-precluster tilting triple} and conditions (a) and (c) and the ``if" part of condition (b) are straightforward to check. 
Property (5) follows from the definition of $\K$ and the fact that a short exact sequence of $\P$-modules gives rise to a triangle in $\D^b(\mod\P)\iso \tr(\K)$. % items (i) and (iv) in Definition~\ref{def:d-precluster tilting}.
Property (2) is immediate because $\I=\inj\Lambda$ has an additive generator.

It remains to show properties (3) and (4) and the ``only if" part of condition (b). 
We show them in the $d=1$ case and $d>1$ case separately.
Since $\P$ is contravariantly finite in $\mod\Lambda$ and contains $\Lambda$, the functor 
\[
F:\mod\Lambda\rightarrow \mod\P,\quad X\mapsto \Hom_{\mod\Lambda}(-,X)|_{\P}
\]
is fully faithful.
Each object $K$ in $\K$ is of the form $F(K_{d+1})$ for some complex 
in $\mod\Lambda$
\begin{equation}\label{cplx:K1}
\begin{tikzcd}
0\ar[r]&K_{d+1}\ar[r,"f_{d+1}"]&P_d\ar[r,"h_d"]&P_{d-1}\ar[r]&\ldots\ar[r]&P_2\ar[r]&P_1\ar[r]&P_0,
\end{tikzcd}
\end{equation}
where $P_i\in\P$ for $0\leq i\leq d$, which remains exact after applying the functor $\Hom_{\mod\Lambda}(M,-)$.
%Property (6) will follow by Corollary~\ref{cor:abelian}.

Suppose $d=1$. Then the objects in $\K$ are precisely the $\P$-modules $\Hom_{\mod\Lambda}(-,K)|_{\P}$ 
where $K\in\mod\Lambda$: each $\Lambda$-module $K$ can be realized as 
the kernel of a morphism between objects in $\P$ since $\P$ contains an injective cogenerator.
The ``only if" part of condition (b) is straightforward to check.
Let $K$ be any object in $\mod\Lambda$.
By items (i) and (iv), we have an exact sequence in $\mod\Lambda$
\[
\begin{tikzcd}
0\ar[r]&K'\ar[r,tail]&P\ar[r,two heads]&K\ar[r]&0,
\end{tikzcd}
\]
where $P\in\P=\add M$, which remains exact after applying the functor $\Hom_{\mod\Lambda}(M,-)$.
So any morphism $K'\rightarrow \tau M$ factors through the morphism $K'\rightarrow P$, cf.~\cite[Chapter III, Corollary 4.2]{Auslander76}.
Since $\tau(\P)\subseteq \P$ and $\tau^{-}(\P)\subseteq \P$, the functor $\tau:\underline{\P}\rightarrow \overline{\P}$ is an equivalence of categories. 
It follows that any morphism $K'\rightarrow M$ factors through the morphism $K'\rightarrow P$.
Therefore we have $\Ext^{\geq 1}_{\mod\P}(\Hom_{\mod\Lambda}(-,K)|_{\P},\Hom_{\mod\Lambda}(-,M))=0$.
This proves property (4). Property (3) also follows by applying the Horseshoe Lemma. 

Suppose $d>1$. Then we have ${^{\perp_{d-1}}\P}={\P^{\perp_{d-1}}}$.
Consider a complex (\ref{cplx:HK2}) in $H^0(\K)$ such that the corresponding complex (\ref{cplx:K2}) is exact for each $P\in \P$. 
Since $G$ is fully faithful, it is given by a complex (\ref{cplx:K1}) under the functor $G$
where $P_i\in\P$ for $0\leq i\leq d$, which remains exact after applying the functor $\Hom_{\mod\Lambda}(M,-)$.
Since $\Lambda$ is contained in $\P$, the complex (\ref{cplx:K1}) is exact in $\mod\Lambda$. 
Since $\I$ consists of injective $\Lambda$-modules and $G$ is fully faithful, we see that the corresponding complex (\ref{cplx:hom2}) is exact for each $I\in\I$.
This shows the ``only if" part of condition (b).
Let us show property (4).
Let $K$ be an object in $\K$. So it is of the form $F(K_{d+1})$ for some complex (\ref{cplx:K1})
in $\mod\Lambda$
where $P_i\in\P$ for $0\leq i\leq d$, which remains exact after applying the functor $\Hom_{\mod\Lambda}(M,-)$.
Then we have $K_{d+1}\in {\P^{\perp_{d-1}}}$. 
Since we have ${^{\perp_{d-1}}\P}={\P^{\perp_{d-1}}}$, the object $K_{d+1}$ lies in ${^{\perp_{d-1}}\P}$.
Since $\P=\add M$ is contravariantly finite in $\mod\Lambda$, we have an exact sequence in $\mod\Lambda$
\[
\cdots \rightarrow P_{d+3} \xrightarrow{g_{d+2}} P_{d+2}\xrightarrow{g_{d+1}} P_{d+1}\xrightarrow{g_{d}} K_{d+1}\rightarrow 0
\] 
which is exact after applying the functor $G$.
Similarly we have that $\Im (g_i)\in {^{\perp_{d-1}}\P}$ for $i\geq d+1$.
It follows that for each $P\in\P$ we have
\[
\Ext^{\geq 1}_{\mod\P}(\Hom_{\mod\Lambda}(-,K_{d+1})|_{\P},\Hom_{\mod\Lambda}(-,P))=0.
\]
This shows property (4).
By condition (iv), we have a left $\P$-approximation $K_{d+1}\rightarrowtail Q_{d}$ with $L_d$ its cokernel. 
Again we have a  left $\P$-approximation $L_{d}\rightarrowtail Q_{d-1}$.  We repeat the above procedure and we have a complex in $\mod\Lambda$
\begin{equation}\label{cplx:K2'}
\begin{tikzcd}
0\ar[r]&K_{d+1}\ar[r]&Q_d\ar[r]&Q_{d-1}\ar[r]&\ldots\ar[r]&Q_2\ar[r]&Q_1\ar[r]&Q_0,
\end{tikzcd}
\end{equation}
where $Q_i\in\P$ for $0\leq i\leq d$.
Since $K_{d+1}\in {\P^{\perp_{d-1}}}$, the above complex remains exact after applying the functor $\Hom(M,-)$.
So we may replace the complex (\ref{cplx:K1}) by (\ref{cplx:K2'}). Since $f_{d+1}$ is a left $\P$-approximation, we have 
\[
\Ext^{i}_{\mod\P}(\Hom_{\mod\Lambda}(-,L_{d})|_{\P},\Hom_{\mod\Lambda}(-,M))=0.
\]
for $1\leq i\leq d-1$.
So we may apply the Horseshoe Lemma and the dg category $\K$ is extension-closed in $\pretr(\K)\iso\D^b_{dg}(\mod\P)$.
This shows property (3) and hence finishes the proof.
\end{proof}
\begin{lemma}
Assume the triple $(\K,\P,\I)$ is $d$-precluster tilting and satisfies the conditions (a), (b) and (c) in Proposition~\ref{prop:Abelian}. 
Then the $\Lambda$-module $M$ is $d$-precluster tilting.
\end{lemma}
\begin{proof}
Item (iv) in Definition~\ref{def:d-precluster tilting} is straightforward.
Since $\P=\add M$ contains the injective $\Lambda$-modules, it is a cogenerator in $\mod\Lambda$.
Let $I$ be an object in $\I$. 
We show that $I^{\wedge}$ is injective in $\mod\P$.
Let $X$ be an object in $\mod\P$.
By condition (c) in Proposition~\ref{prop:Abelian} (and condition (a) if $d=1$), we have an exact sequence in $\mod\P$
\[
\begin{tikzcd}
P_2^{\wedge}\ar[r]&P_1^{\wedge}\ar[r]&P_0^{\wedge}\ar[r]&X\ar[r]&0.
\end{tikzcd}
\]
By applying the functor $\Hom_{\mod\P}(-,I^{\wedge})$, we get a sequence of Hom spaces in $\mod\P$
\[
\begin{tikzcd}
0\ar[r]&\Hom(X,I^{\wedge})\ar[r]&\Hom(P_0^{\wedge},I^{\wedge})\ar[r]&\Hom(P_1^{\wedge},I^{\wedge})\ar[r]&\Hom(P_2^{\wedge},I^{\wedge})
\end{tikzcd}
\]
which is exact by condition (b). It follows that $\Ext^{1}_{\mod\P}(X,I^{\wedge})=0$ and hence $I^{\wedge}$ is injective in $\mod\P$.
Let us show that $\P$ also contains the projective $\Lambda$-modules.
The canonical morphism $\Hom_{\mod\Lambda}(\Lambda,I)\rightarrow \Hom_{\mod\P}(\Hom_{\mod\Lambda}(-,\Lambda)|_{\P},I^{\wedge})$ is an isomorphism for $I\in \I$. Indeed, since $I^{\wedge}$ is injective, it is isomorphic to $D\Hom_{\mod\Lambda}(P,-)$ for some $P\in\P$. Then we have 
\begin{align*}
\Hom_{\mod\P}(\Hom_{\mod\Lambda}(-,\Lambda)|_{\P},I^{\wedge})&\iso \Hom_{\mod\P}(\Hom_{\mod\Lambda}(-,\Lambda)|_{\P},D\Hom(P,-))\\&\iso D\Hom(P,\Lambda)\\&\iso \Hom(\Lambda,I).
\end{align*}
We have a sequence in $\mod\Lambda$
\begin{equation}\label{seq:exact}
\begin{tikzcd}
Q_1\ar[r,"g"]&Q_0\ar[r,"f"]&\Lambda\ar[r]&0,
\end{tikzcd}
\end{equation}
where $Q_1$ and $Q_0$ are in $\P$, which becomes exact after applying the functor $\Hom_{\mod\Lambda}(M,-)$.
Since $I^{\wedge}$ is injective in $\mod\Lambda$ for each injective $I\in\mod\Lambda$, by the above we have an exact sequence
\[
\begin{tikzcd}
0\ar[r]&\Hom(\Lambda,I)\ar[r]&\Hom(Q_0,I)\ar[r]&\Hom(Q_1,I).
\end{tikzcd}
\]
It follows that the sequence (\ref{seq:exact}) is exact in $\mod\Lambda$ and hence $\Lambda$ is a direct summand of $Q_0$.
Therefore we have $\Lambda\in\P$ and this proves item (i) in Definition~\ref{def:d-precluster tilting}. 
As a consequence, the functor 
\[
G:\mod\Lambda\rightarrow \mod\P,\quad X\mapsto \Hom_{\mod\Lambda}(-,X)|_{\P}
\]
is fully faithful.
%If $d>1$, we have an exact sequence in $\mod\Lambda$
%\[
%\begin{tikzcd}
%0\ar[r]&K_{d+1}\ar[r]&P_{d}\ar[r]&P_{d-1}\ar[r]&\ldots\ar[r]&P_3\ar[r]&P_2\ar[r]&\Lambda,
%\end{tikzcd}
%\]
%where $P_i\in\P$ for $2\leq i\leq d-1$, which remains exact after applying the functor $\Hom(M,-)$.
%We also have an exact sequence $0\rightarrow \Lambda\rightarrowtail P_{1}\rightarrow P_0$ in $\mod\Lambda$.
%Then $\Hom_{\mod\P}(-,K_{d+1})|_{\P}$ is an object in $\K$ and hence by condition (b) of Proposition~\ref{prop:Abelian}, the morphism $M_2\rightarrow \Lambda$ is a surjection. 

Let us show item (iii). 
We take the first $d$-term of an injective coresolution of $M$ in $\mod\Lambda$
\begin{equation}\label{seq:coresolution}
\begin{tikzcd}
0\ar[r]&M\ar[r]&I_d\ar[r]&I_{d-1}\ar[r]&\ldots\ar[r]&I_1\ar[r]&I_0.
\end{tikzcd}
\end{equation}
By condition (b) in Proposition~\ref{prop:Abelian}, the sequence (\ref{seq:coresolution}) remains exact after applying the functor $\Hom_{\mod\Lambda}(M,-)$.
So by definition we have $\Ext^{i}(M,M)=0$ for $0<i<d$.

Finally let us show item (ii). We will show this in $d=1$ case and $d>1$ case separately.
Assume $d=1$.
Let us show that $\tau^{-}(\P)\subset \P$. 
Since $\P=\add M$ contains the injectives and the projectives, 
and the number of isoclasses of indecomposable injectives 
equals the number of isoclasses of indecomposable projectives, 
it follows that $\tau(\P)\subseteq \P$.
Let $N$ be an indecomposable noninjective $\Lambda$-module in $\P$.
We have $\Ext^{\geq 1}_{\mod\P}(G(\tau^{-1}(N)),G(P))=0$ for $P\in\P$ by property (4) in Definition~\ref{def:d-precluster tilting triple}.
We have a sequence in $\mod\Lambda$
\begin{equation}\label{seq:tauN}
\cdots\rightarrow P^2\rightarrow P^1\xrightarrow{u} P^0\xrightarrow{v} \tau^{-}(N)\rightarrow 0,
\end{equation}
where $P^i\in\P$ for $i\geq 0$, and which is exact after applying the functor $G$. 
So after applying the functor $G$, it is a projective resolution of $G(\tau^{-1}(N))$. 
It follows that the above sequence (\ref{seq:tauN}) is exact after applying the functor $\Hom_{\mod\Lambda}(-,P)$ for each $P\in\P$.
Put $L=\Im(u)$. 
Then the inclusion of $L$ into $P^0$ induces an epimorphism $\Hom(P^0,P)\rightarrow \Hom(L,P)$ for each $P\in\P$.
Hence the morphism $v:P^0\rightarrow \tau^{-}(N)$ induces an epimorphism $\Hom(\tau^{-}(P),P^0)\rightarrow \Hom(\tau^{-}(P),\tau^{-}(N))$.
In particular, the morphism $v$ splits and $\tau^{-}(N)$ is a direct summand of $P^0\in \P$. 

Assume $d>1$. 
Let us show that $\tau_{d}^{-}(\P)\subseteq \P$. 
By \cite[Theorem 1.4.1]{Iyama07a}, it follows that $\tau_{d}(\P)\subseteq \P$.
By Lemma~\ref{lem:taud}, it suffices to show ${^{\perp_{d-1}}\P}\supseteq \P^{\perp_{d-1}}$.
Let $Y\in \P^{\perp_{d-1}}$.
Then we have an exact sequence in $\mod\Lambda$
\[
0\rightarrow Y\rightarrow Q_d\rightarrow Q_{d-1}\rightarrow \cdots\rightarrow \Q_1\rightarrow Q_0,
\]
where $Q_d\in \P$, and which is exact after applying the functor $\Hom(-,M)$.
Since $Y\in \P^{\perp_{d-1}}$, the above sequence remains exact after applying the functor $\Hom(P,-)$ for each $P\in\P$.
Hence $G(Y)$ lies in $\K$.
By property (4) of Definition~\ref{def:d-precluster tilting triple}, we have $\Ext^{\geq 1}_{\mod\P}(G(Y), G(P))=0$ for each $P\in \P$.
We also have an exact sequence in $\mod \Lambda$
\[
\cdots \rightarrow Q_{d+3} \xrightarrow{g_{d+2}} Q_{d+2}\xrightarrow{g_{d+1}} Q_{d+1}\xrightarrow{g_{d}} Y\rightarrow 0
\]
where $Q_i\in \P$ for $i\geq d+1$, and which remains exact after applying the functor $\Hom(P,-)$ for each $P\in \P$. So we get a projective resolution of $G(Y)$ after applying the functor $G$.
Hence the above sequence remains exact after applying the functor $\Hom(-,P)$ for each $P\in \P$.
Since $\Ext^i_{\Lambda}(\P,\P)=0$ for $0<i<d$, it follows that $Y\in {^{\perp_{d-1}}\P}$.
\end{proof}
\subsection{Connection to higher Auslander--Solberg correspondence for exact categories}\label{subsec:Grevstad}
In the rest of this section, we assume $d\geq 1$ and compare our correspondence with the correspondence obtained by Grevstad in \cite{Grevstad22}.
By combining Proposition~\ref{prop:generalizationpre} and Lemma~\ref{lem:generalizationpre}, we see that we get a generalization of the correspondence \cite[Theorem 3.8]{Grevstad22} in the case of weakly idempotent complete exact categories with enough injectives.
\begin{definition}[{\cite[Proposition 1.7]{DraxlerReitenSmaloSolberg99}}]\label{def:relativestructure}
Let $\M\subseteq \E$ be a full subcategory of an exact category. 
The relative exact structure $F_{\M}$ is given by the conflations in $\E$ which remain exact after applying the functor $\Hom_{\E}(M,-)$ for any $M\in \M$.
\end{definition}
%\begin{definition}[{\cite[Definition 8]{ZhuZhuang20}}]
%Let $\T\subseteq\E$ be a full additive subcategory of an exact category which is closed under under direct summands. It is {\em $d$-cotilting} if 
%\begin{itemize}
%\item it is functorially finite in $\E$;
%\item the injective dimension of $\T$ is less than or equal to $d$ for every 
%\item
%\end{itemize}
%\end{definition}
\begin{definition}[\cite{Grevstad22}]\label{def:d-precluster tilting exact}
Let $\M$ be a subcategory of an exact category $\E$. 
We say that $\M$ is {\em partial $d$-precluster tilting} if it satisfies
\begin{itemize}
\item[(a)] every object $E\in\E$ has a left $\M$-approximation by an inflation $E\rightarrowtail M$;
\item[(b)] every object $E\in\E$ has a right $\M$-approximation by a deflation $M'\twoheadrightarrow E$
\item[(c)] $\M$ is $d$-rigid, i.e.~$\Ext^{1\sim d-1}_{\E}(\M,\M)=0$;
\item[(d)] the relative injective dimension $\Id_{F_{\M}}M$ is strictly less than $d$ for any $M\in\M$,  where $F_{\M}$ is the exact substructure defined in Definition~\ref{def:relativestructure}; 
\end{itemize}
%It is {\em $d$-precluster tilting} if furthermore it satisfies
%\begin{itemize}
%\item[(e)] $\M$ is a cogenerator in ${^{\perp_{\infty}}\M}=\{Y\in \E\mid \Ext^{>0}_{F_{\M}}(Y,\M)=0, \forall M\in\M\}$, i.e.~for every $X\in{^{\perp_{\infty}}\M}$, there is a conflaiton
%\[
%\begin{tikzcd}
%X\ar[r, tail]&M\ar[r,two heads]&X'
%\end{tikzcd}
%\]
%where $M\in \M$ and $X'\in {^{\perp_{\infty}}\M}$.
%\end{itemize} 
\end{definition}
A morphism $f:X\rightarrow Y$ in an exact category $\E$ is {\em admissible} if $f$ admits a factorization $X\twoheadrightarrow I\rightarrowtail Y$.
A cochain complex over $\E$
\[
\ldots\rightarrow X^{n-1}\xrightarrow{d^{n-1}} X^n\xrightarrow{d^n} X^{n+1}\rightarrow \ldots
\]
is called {\em acyclic} or {\em exact} if each $d^i$ is admissible and $\ker(d^{i+1})=\im(d^i)$.
Let $\M$ be a full subcategory of $\E$.
We write $\mod_{adm}(\M)$ \cite[Definition 3.1 (2)]{EbrahimiNasrIsfahani21} for the full subcategory of $\Mod \M$ consisting of those functors $F$ that admit a projective presentation
\[
\begin{tikzcd}
M^{\wedge}\ar[r,"f^{\wedge}"]&N^{\wedge}\ar[r,two heads] &F\ar[r]&0
\end{tikzcd}
\]
for some morphism $f:M\rightarrow N$ in $\M$ which is admissible in $\E$.
Recall that for a full subcategory $\X$ of $\E$, we write 
\[
{^{\perp}\X}=\{E\in\E\mid\Hom_{\E}(E,X)=0\text{ for any $X\in\X$}\}.
\]
Put $\cogen(\X)=\{E\in\E\mid \exists \text{ inflation } E\rightarrowtail X \text{ with } X\in \X\}$.
A pair $(\T,\F)$ of full subcategories of $\E$ is a {\em torsion pair} \cite[Definition 2.17]{HenrardKvammevanRoosmalen22} if
\begin{itemize}
\item[(1)] For every $T\in\T$ and $F\in\F$, we have $\Hom_{\E}(T,F)=0$.
\item[(2)] For any object $M\in\E$, there exists a conflation $T\rightarrowtail M\twoheadrightarrow F$ with $T\in\T$ and $F\in \F$.
\end{itemize} 

\begin{theorem}[{\cite{Grevstad22}}]\label{thm:Grevstad}
{\rm
There is a bijection between the following:
\begin{itemize}
\item[(1)] Equivalence classes of partial $d$-precluster tilting subcategories  $\M$ of exact categories $\E$ with enough injectives.
\item[(2)] Equivalence classes of exact categories $\E'$ with enough projectives, where we denote by $\P=\proj(\E')$ the full subcategory of projectives, satisfying the following conditions 
\subitem{($a$)} $\dom(\E')\geq d+1$.
\subitem($b$) Each projective $P$ has injective dimension $\leq d+1$, i.e.~$\Ext^{\geq d+2}_{\E'}(-,P)=0$.
\subitem($c$) Any morphism $X\rightarrow E$ with $E\in {^{\perp}\P}$ is admissible.
\subitem($d$) $\P$ is {\em admissibly covariantly finite}, i.e.~for each object $E$ in $\E$, there is a left $\P$-approximation $E\rightarrow P$ which is admissible.
\end{itemize}
The bijection from (1) to (2) sends $(\E,\M)$ to $\mod_{adm}(\M)$.
}
\end{theorem}
\begin{remark}
The above bijection is a restriction of the bijection in \cite[Theorem 3.8]{Grevstad22}. More precisely, in Proposition 6.3 in {\em loc.~cit}, it is shown that in the bijection of Theorem 3.8 in {\it loc.~cit}, the exact categories $\E$ has enough injectives if and only if for the corresponding exact category $\E'$ we have $\dom(\E')\geq 1$.
But since $\M$ is $d$-rigid, we have $\dom(\E')\geq d+1$ in this case.
By \cite[Lemma 4.17 (2)]{HenrardKvammevanRoosmalen22}, if $\E'$ has enough projectives and if $\P$ is admissibly covariantly finite and $\dom(\E')\geq 1$, then $({^{\perp}\P}, \cogen(\P))$ is a torsion pair. 
By \cite[Lemma 4.15]{HenrardKvammevanRoosmalen22}, if $\E'$ has enough projectives and $\dom(\E')\geq 1$, then $ {^{\perp}\P}$ is closed under admissible subobjects.
Since $\dom(\E')\geq d+1$, a similar proof of \cite[Lemma 4.14]{HenrardKvammevanRoosmalen22} shows that $\Ext^{1\sim d}_{\E}({^{\perp}\P},\P)=0$.
Therefore the exact categories appearing in (2) of the above bijection are partial $d$-minimal Auslander--Gorenstein in the sense of \cite[Definition 3.4]{Grevstad22}.
\end{remark}

Let $\E$ be a weakly idempotent complete exact category with enough injectives.
Let $\M$ be a $d$-precluster tilting subcategory and $\I$ the full subcategory of $\E$ on the injectives.
We define $\K$ to be the full dg subcategory $\D^b_{dg}(\Mod\M)$ consisting of the right $\M$-modules $\Hom_{\E}(-,K)|_{\M}$ where $K$ appears in an exact sequence
\begin{equation}\label{exactseq:K}
\begin{tikzcd}
0\ar[r]&K\ar[r]&M_d\ar[r,"f_d"]&M_{d-1}\ar[r]&\ldots\ar[r]&M_2\ar[r]&M_1\ar[r,"f_1"]&M_0
\end{tikzcd}
\end{equation}
where $M_i\in\M$ for $0\leq i\leq d$ and which remains exact after applying the functor $\Hom_{\E}(M,-)$ for each $M\in\M$.
 We identify $\M$ with the full subcategory of representable $\M$-modules.
 By Proposition~\ref{prop:generalizationpre} and Lemma~\ref{lem:generalizationpre} below, we have the following diagram
 
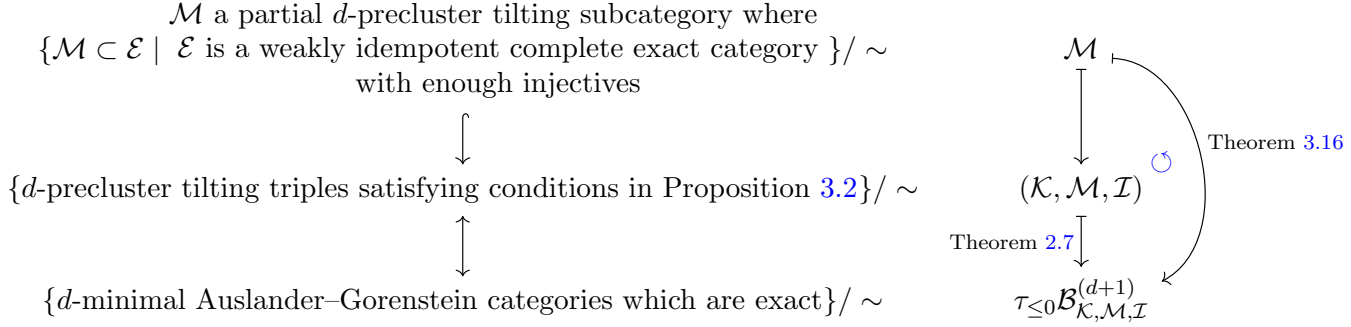
\[
\begin{tikzcd}
\{\M\subset\E \mid \begin{matrix}\text{$\M$ a partial $d$-precluster tilting subcategory where }\\\text{$\E$ is a weakly idempotent complete exact category}\\ \text{with enough injectives}\end{matrix}\}/\sim\ar[d,hook]&\M\ar[d,mapsto]\ar[dd,mapsto, bend left=16ex,"\text{Theorem~\ref{thm:Grevstad}}"]\ar[dd,phantom,"\circlearrowleft"blue, bend left=11ex]\\
\{\text{$d$-precluster tilting triples satisfying conditions in Proposition~\ref{prop:Quillenexact}}\}/\sim\ar[d,<->]&(\K,\M,\I)\ar[d,mapsto,"\text{Theorem~\ref{thm:Iyama--Solberg correspondence}}"swap]\\
\{\text{$d$-minimal Auslander--Gorenstein categories which are exact}\}/\sim&\tau_{\leq 0}\B_{\K,\M,\I}^{(d+1)}
\end{tikzcd}
\]
\begin{proposition}\label{prop:generalizationpre}
The triple $(\K,\M,\I)$ is $d$-precluster tilting and satisfies the conditions in Proposition~\ref{prop:Quillenexact}.
\end{proposition}
\begin{proof}
The inclusion $\K\hookrightarrow \D^b_{dg}(\Mod\M)$ induces an inclusion $\pretr(\K)\hookrightarrow\D^b_{dg}(\Mod\M)$.
It is clear that $\K$ is concentrated in non-negative degrees.
So condition (a) in Proposition~\ref{prop:Quillenexact} and property (1) in Definition~\ref{def:d-precluster tilting triple} hold.

By Definition~\ref{def:d-precluster tilting exact} (b), the functor $F: \E\rightarrow \Mod\M$ sending $X\in\E$ to $\Hom_{\E}(-,X)|_{\M}$ is fully faithful.
We thus identify $H^0(\K)$ with the corresponding full subcategory of $\E$.
By definition of $\K$ and that $\M$ is $d$-rigid, we have that $\Ext^{1\sim d-1}_{\E}(M,K)=0$ for $M\in \M$ and $K\in H^0(\K)$.
Let us show that $\Ext^1_{F_{\M}}(K,M)=0$ for $K\in H^0(\K)$ and $M\in \M$.
Indeed, we have conflations in $F_{\M}$
\[
\begin{tikzcd}
K\ar[r, tail]&N_d\ar[r,two heads]&V_d
\end{tikzcd}
\]
\[
\begin{tikzcd}
V_d\ar[r, tail]&N_{d-1}\ar[r,two heads]&V_{d-1}
\end{tikzcd}
\]
\[
\ldots
\]
\[
\begin{tikzcd}
V_3\ar[r, tail]&N_2\ar[r,two heads]&V_2
\end{tikzcd}
\]
where $N_i\in \M$ for $2\leq i\leq d$.
Then we have 
\[
\Ext^1_{F_{\M}}(K,M)\iso \Ext^2_{F_{\M}}(V_d,M)\iso\ldots \iso \Ext^d_{F_{\M}}(V_2,M)=0
\]
since the relative injective dimension $\Id_{F_{\M}}M$ is strictly less than $d$ for $M\in \M$. 

Now we show that $\Ext^{\geq 1}_{\Mod\M}(F(K), F(M))=0$ for $K\in H^0(\K)$ and $M\in \M$ and so property (4) in Definition~\ref{def:d-precluster tilting triple} holds.
By Definition~\ref{def:d-precluster tilting exact} (b), there exists a conflation in $\E$
\begin{equation}\label{confl:K}
\begin{tikzcd}
K'\ar[r,tail,"i"]&M'\ar[r,two heads,"p"]&K
\end{tikzcd}
\end{equation}
where $p$ is a right $\M$-approximation of $K$. It follows that $K'$ also lies in $H^0(\K)$. Therefore property (5) in Definition~\ref{def:d-precluster tilting triple} holds. Now the conflation (\ref{confl:K}) also lies in $F_{\M}$. Since $\Ext^1_{F_{\M}}(K,M)=0$, the morphism $i$ induces a surjection $\Hom_{\E}(M',M)\rightarrow\Hom_{\E}(K',M)$.
So we have an exact sequence in $F_{\M}$
\[
\begin{tikzcd}
\ldots\ar[r]&M^j\ar[r]&\ldots\ar[r]&M^1\ar[r]&M^0\ar[r,two heads]&K
\end{tikzcd}
\]
which remains exact after applying the functor $\Hom_{\E}(-,M)$ for each $M\in\M$. 
After applying the functor $F:\E\rightarrow \Mod\M$, we obtain a projective resolution of $F(K)$. Since the functor $F$ is fully faithful when restricted to $H^0(\K)$, we conclude that $\Ext^{\geq 1}_{\Mod\M}(F(K),F(M))=0$.

Condition (b) in Proposition~\ref{prop:Quillenexact} can be checked as follows: consider a complex
\begin{equation}\label{cplx:exact1}
\begin{tikzcd}
0\ar[r]&K_{d+1}\ar[r]& M_d\ar[r]&M_{d-1}\ar[r]&\ldots\ar[r]&M_1\ar[r]&M_0
\end{tikzcd}
\end{equation}
where $K_{d+1}\in H^0(\K)$ and $M_i\in \M$ for $0\leq i\leq d$.
Suppose that it remains exact after applying the functor $\Hom_{\E}(-,I)$ for each $I\in\I$.
Since $\E$ has enough injectives and is weakly idempotent complete, we see that the morphism $K_{d+1}\rightarrow M_d$ is an inflation which is part of a conflation in $\E$
\[
\begin{tikzcd}
K_{d+1}\ar[r,tail]&M_d\ar[r,two heads]&U_d.
\end{tikzcd}
\]
Then the map $M_d\rightarrow M_{d-1}$ factors through $U_d$ uniquely and similarly the map $U_d\rightarrow M_{d-1}$ is an inflation. We repeat the above argument and we see that the above sequence (\ref{cplx:exact1}) is exact. 
Since $\M$ is $d$-rigid and $\Ext^{1\sim d-1}_{\E}(M,K_{d+1})=0$ for $M\in \M$, we see that the sequence (\ref{cplx:exact1}) remains exact after applying the functor $\Hom_{\E}(M,-)$ for each $M\in\M$.

Let $K$ be any object in $H^0(\K)$. 
Since $\E$ has enough injectives, we have an exact sequence
\[
\begin{tikzcd}
0\ar[r]&K\ar[r]& I_d\ar[r]&I_{d-1}\ar[r]&\ldots\ar[r]&I_1\ar[r]&I_0
\end{tikzcd}
\]
where $I_i\in \I$ for $0\leq i\leq d$.
By the above, we see that the above sequence remains exact after applying the functor $\Hom(M,-)$ for each $M\in \M$.
Hence, in the sequence (\ref{exactseq:K}), we could assume that $M_i\in \I$ for $0\leq i\leq d$.
Therefore, by the Horseshoe Lemma we see that $\K$ is extension-closed in $\pretr(\K)\subseteq \D^b_{dg}(\Mod\M)$ and so property (3) in Definition~\ref{def:d-precluster tilting triple} holds.

Let us show property (2), i.e.~$H^0(\I)$ is covariantly finite in 
\[
\B_{\K,\M,\I}^{(d)}=\K\ast\Sigma\K\ast\ldots\ast\Sigma^d\K=\M\ast\Sigma\M\ast\ldots\ast\Sigma^{d-1}\M\ast\Sigma^{d}\K.
\]
Objects in $\B_{\K,\M,\I}^{(d)}$ are given by exact sequences in $\E$ 
\begin{equation}\label{exactseq:approx}
\begin{tikzcd}
0\ar[r]&K\ar[r,"f_{d+1}"]&M_d\ar[r,"f_d"]&M_{d-1}\ar[r]&\ldots\ar[r]&M_2\ar[r,"f_2"]&M_1
\end{tikzcd}
\end{equation}
 where $K\in H^0(\K)$ and $M_i\in \M $ for $0\leq i \leq d$. Indeed, the above sequence splits into conflations
 \[
 \begin{tikzcd}
 K\ar[r,tail]&M_d\ar[r,two heads] &U_d,
 \end{tikzcd}
 \]
 \[
 \begin{tikzcd}
 U_d\ar[r,tail]&M_{d-1}\ar[r,two heads] &U_{d-1},
 \end{tikzcd}
 \]
 \[
 \ldots
 \]
  \[
 \begin{tikzcd}
 U_3\ar[r,tail]&M_2\ar[r,two heads] &U_2.
 \end{tikzcd}
 \]
 And the object in $\B_{\K,\M,\I}^{(d)}$, which corresponds to the exact sequence (\ref{exactseq:approx}), is given by the cokernel of the morphism $F(U_2)\rightarrow F(M_1)$.
 Since the morphism $f_2:M_2\rightarrow M_1$ is admissible, we take its cokernel $M_1\rightarrow L$ and we have an inflation $L\rightarrowtail I$ for some $I\in \I$.
Then the composition $M_1\rightarrow I$ gives rise to a morphism from the sequence (\ref{exactseq:approx}) to $I$. It is straightforward to  check that it is a left $\I$-approximation.
\end{proof}

\begin{lemma}\label{lem:generalizationpre}
As above, the functor $F:\B'=\tau_{\leq 0}\B_{\K,\M,\I}^{(d+1)}\rightarrow \Mod\M$ in Lemma~\ref{lem:Quillenexact} induces an equivalence of exact categories $F: \B'\rightarrow \mod_{adm}(\M)$.
\end{lemma}
\begin{proof}
Recall that we identify $H^0(\K)$ with the corresponding subcategory of $\E$ via the functor $\E\rightarrow \Mod\M$ given by sending $E\in \E$ to $\Hom_{\E}(-,E)|_{\M}$.
Let $X$ be an object in $\B'$ with conflations (\ref{conf:X0}) where $K_{d+1}\in H^0(\K)$ and $P_i\in \M$ for $0\leq i\leq d$.
Then the object $F(X)\in \Mod \P$ is given by the cokernel of the morphism $P_1^{\wedge}\rightarrow P_0^{\wedge}$.
It suffices to show that the morphism $P_1\rightarrow P_0$ is admissible in $\E$.
We can check this in a same way as we check condition (b) in the proof of Proposition~\ref{prop:generalizationpre}.
Therefore the functor $F$ gives rise to an exact functor $F:\B'\rightarrow \mod_{adm}(\M)$.
Conversely, for each object $U$ in $\mod_{adm}(\M)$, it is the cokernel of $P_1^{\wedge}\xrightarrow{f^{\wedge}} P_0^{\wedge}$ for some admissible morphism $f: P_1\rightarrow P_0$ in $\M$.
By Definition~\ref{def:d-precluster tilting exact} (2), we may complete it to an exact sequence in $\E$
\begin{equation}
\begin{tikzcd}\label{exact:K}
0\ar[r]&K_{d+1}\ar[r]& P_d\ar[r]&P_{d-1}\ar[r]&\ldots\ar[r]&P_1\ar[r]&P_0
\end{tikzcd}
\end{equation}
where $P_i\in \M$ for $0\leq i\leq d$.
It follows that $K_{d+1}\in H^0(\K)$ and the exact sequence (\ref{exact:K}) gives rise to conflations (\ref{conf:X0}) in $\B'$.
By definition, we have that $F(X)$ is isomorphic to the cokernel of $P_1^{\wedge}\xrightarrow{f^{\wedge}} P_0^{\wedge}$, which is in turn isomorphic to $U$.
\end{proof}

	\def\cprime{$'$} \def\cprime{$'$}
	\providecommand{\bysame}{\leavevmode\hbox to3em{\hrulefill}\thinspace}
	\providecommand{\MR}{\relax\ifhmode\unskip\space\fi MR }
	\providecommand{\MRhref}[2]{%
		\href{http://www.ams.org/mathscinet-getitem?mr=#1}{#2}
	}
	\providecommand{\href}[2]{#2}
	%\begin{thebibliography}{10}
%		

	%\end{thebibliography}

	\bibliographystyle{amsplain}
	\bibliography{stanKeller}

\end{document}